\documentclass[a4paper]{lipics-v2021}% add anonymous to the []

\usepackage{mathpartir}
\usepackage{graphicx,graphbox}
\usepackage{multicol}

\pdfoutput=1 %uncomment to ensure pdflatex processing (mandatatory e.g. to submit to arXiv)
\hideLIPIcs  %uncomment to remove references to LIPIcs series (logo, DOI, ...), e.g. when preparing a pre-final version to be uploaded to arXiv or another public repository
 
\title{Override and Update in Restriction Categories}

\author{Jean-Simon Pacaud Lemay}{Macquarie University, Australia}{js.lemay@mq.edu.au}{}{This material is based upon work supported by the AFOSR under award number FA9550-24-1-0008.}

\author{Chad Nester}{University of Tartu, Estonia}{nester@ut.ee}{}{Chad Nester was supported by Estonian Research Council grant PRG2764.}

\authorrunning{J.S.P. Lemay, C. Nester} %TODO mandatory. First: Use abbreviated first/middle names. Second (only in severe cases): Use first author plus 'et al.'

\Copyright{Jean-Simon Pacaud Lemay, Chad Nester} %TODO mandatory, please use full first names. LIPIcs license is "CC-BY";  http://creativecommons.org/licenses/by/3.0/

\ccsdesc[500]{Theory of computation~Categorical semantics}
\ccsdesc[500]{Theory of computation~Logic}

\keywords{Override; Update; Restriction Categories} %TODO mandatory; please add comma-separated list of keywords

\category{} %optional, e.g. invited paper

\relatedversion{} %optional, e.g. full version hosted on arXiv, HAL, or other respository/website
\nolinenumbers %uncomment to disable line numbering

\EventEditors{Corina C\^{\i}rstea and Yannick Forster}
\EventNoEds{2}
\EventLongTitle{35th EACSL Annual Conference on Computer Science Logic (CSL 2027)}
\EventShortTitle{CSL 2027}
\EventAcronym{CSL}
\EventYear{2027}
\EventDate{January 26--29, 2027}
\EventLocation{Brighton, UK}
\EventLogo{}
\SeriesVolume{403}
\ArticleNo{6}

\newcommand{\X}{\mathbb{X}}

\newcommand{\rest}[1]{\overline{#1}}
\newcommand{\override}{\triangleright}
\newcommand{\update}{\diamond}

\newcommand{\blackdiamond}{%
  \mathord{%
    \sbox0{$\diamond$}%
    \resizebox{!}{1.125\ht0}{%
      \raisebox{\depth}{\rotatebox[origin=c]{45}{$\blacksquare$}}%
    }%
  }%
}

\begin{document}
\allowdisplaybreaks

\maketitle

\begin{abstract}
We study the override and update operators on partial functions from the perspective of restriction categories. We propose a definition of override restriction categories, in which both of the operators in question exist. We prove a number of results concerning these operators and their relationship to the structure of the ambient restriction category, as well as relating them to the existing literature on the override and update operators. We provide various examples of override restriction categories and in particular show that every classical restriction category is an override restriction category. 
\end{abstract}  

\section{Introduction}\label{sec:intro}

There are not very many sensible ways to combine arbitrary partial functions with the same domain and codomain. One way is to take their \emph{meet} by viewing the partial functions in question as sets of input-output pairs and taking the intersection of these sets. We might guess that the \emph{join} is defined similarly, by taking the union of associated sets of input-output pairs, but in general this does not yield a partial function. The obstacle is that the union of these sets need not correspond to a function, since a given element of the domain need not be mapped to the same element of the codomain by both of the functions in question.

One way to resolve this difficulty is to consider a sort of biased join operation, in which conflicts of this sort are resolved by doing what, say, the left argument does. This allows us to again obtain a partial function. The resulting binary operation on partial functions is referred to as the \emph{override} operator \cite{Berendsen2010}, or sometimes the (left-)\emph{preferential union} operator. Explicitly, if $f,g : A \to B$ are partial functions then we define a partial function $f \override g : A \to B$ as doing $f$ where $f$ is defined and doing $g$ where $f$ is undefined. There is another closely related binary operator referred to as the \emph{update} operator \cite{Berendsen2010} that sends partial functions $f,g : A \to B$ to the partial function $f \update g : A \to B$, pronounced ``$f$ update $g$'', defined by doing $g$ when both $g$ and $f$ are defined, while doing $f$ when $g$ is undefined and $f$ is defined. 

It is helpful to depict these operations as patterned Venn diagrams (see Figure~\ref{fig:venn-diagrams}). A given partial function is represented by a region of space, corresponding to its domain of definition, which is filled with an associated ``pattern'' or ``colour'' indicating the action of the partial function in question. Then the region of $f \override g$ is the union of the regions of $f$ and $g$, with the pattern of $f$ dominating the pattern of $g$. Similarly, the region of $f \update g$ is the region of $f$, but with the pattern of $g$ dominating the pattern of $f$.

\begin{figure}
\setlength{\tabcolsep}{12pt}
\begin{center}
\begin{tabular}{cccc}
  \includegraphics[height=2cm,align=c]{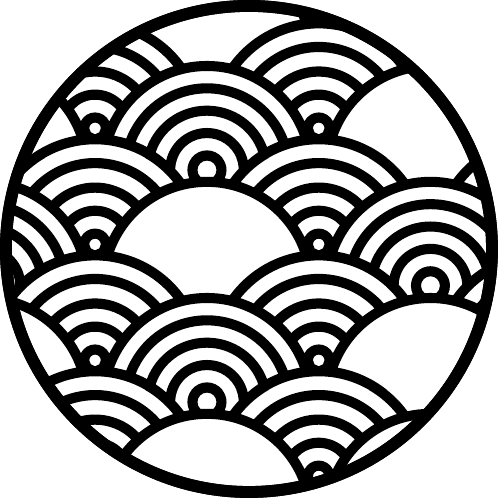}
  &
  \includegraphics[height=2cm,align=c]{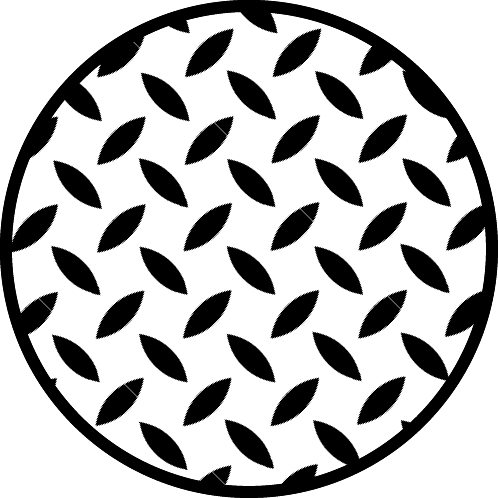}
  &
  \includegraphics[height=2cm,align=c]{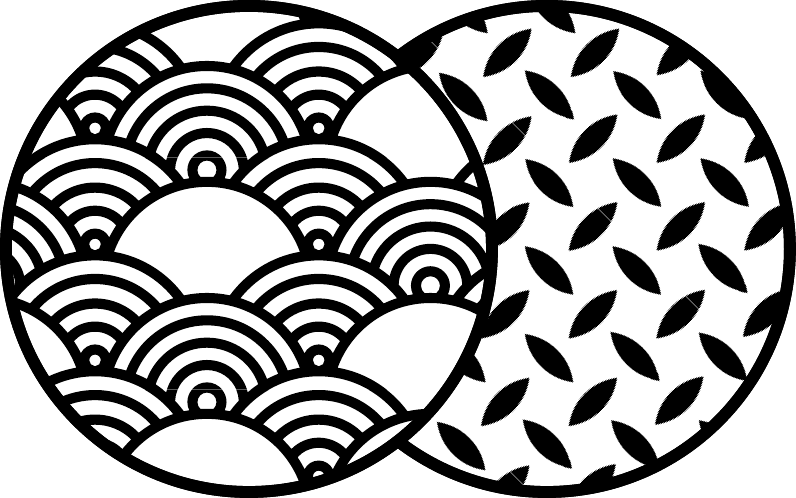}
  &
  \includegraphics[height=2cm,align=c]{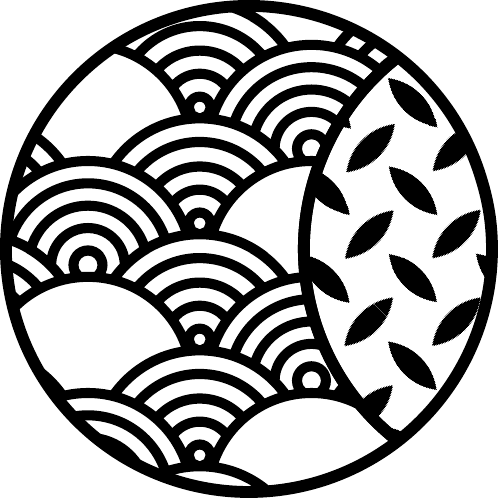}
  \\
  $f : A \to B$
  &
  $g : A \to B$
  &
  $f \override g : A \to B$
  &
  $f \update g : A \to B$
\end{tabular}
\end{center}
\setlength{\tabcolsep}{6pt} % reset to default value
\caption{Patterned Venn diagrams explaining the override and update operators.}\label{fig:venn-diagrams}
\end{figure}

It is algebraically natural to consider the override and update operations together. In particular, Jackson and Stokes have given a complete algebraic axiomatisation of the behaviour of override and update in semigroup theory~\cite{Jackson2021}, building on earlier works of Berendsen et al.~\cite{Berendsen2010} and Cvetko-Vah et al.~\cite{Cvetko2013}. The override and update operators have many applications and are commonly encountered in formal methods for program correctness~\cite{bakker1980mathematical}. For example, the override operator appears in the formal specification languages Z~\cite{Spivey1992} and VDM~\cite{Jones1990}, while the update operator appears in the denotational semantics of assignments~\cite{Berendsen2010}. 

In this paper, we study override and update from the perspective of the theory of restriction categories~\cite{Cockett02}. Restriction categories, which we review in Sec~\ref{sec:restriction}, are abstract categories of partial maps, and provide a natural setting in which to study operations on partial functions. 

In Sec~\ref{sec:override}, we introduce the notion of an override operator for restriction categories, and call a restriction category equipped with an override operator an override restriction category (Def~\ref{def:override-rest-cat}). We provide a list of various basic identities that hold for the override operator (Lem~\ref{lem:override-facts}) and also show that every override restriction category has binary joins (Thm~\ref{thm:override-joins}) in the restriction category sense (Def~\ref{def:join}). 

In Sec~\ref{sec:update} we define the induced update operator (Def~\ref{def:update}) of an override restriction category, where we define the update operator using the override operator and the restriction operator. We then show that every homset of an override restriction category satisfies the axioms for override and update given by Jackson and Stokes~\cite{Jackson2021} (Thm~\ref{thm:ou-algebra}). This helps justify that we have the appropriate definitions for these operators for restriction categories. 

In Sec~\ref{sec:starting-from-update} we take the approach of instead starting by defining the notion of an update operator independently in a restriction category \textit{with binary joins}. We call a restriction with binary joins and an update operator, an update restriction category (Def~\ref{def:update-rest-cat}). We then explain how to define an override operator using the update operator and the join (Prop~\ref{prop:update-to-override}). From this, we are then able to show that in fact, update restriction categories are precisely the same as override restriction categories (Thm~\ref{thm:override=update}). We will also discuss how the behaviour of the override and update operators affect the ambient restriction category. 

We provide various examples of override restriction categories including sets and partial functions (Ex~\ref{ex:par-override} \& \ref{ex:par-update}), rings and non-unital ring morphisms (Ex~\ref{ex:cring-override} \& \ref{ex:cring-update}), and any upper-bounded distributive lattice (Ex~\ref{ex:lattice-override} \& \ref{ex:lattice-update}). In Sec~\ref{sec:classical}, we show that every classical restriction category \cite{Cockett2009Boolean} is an override restriction category (Thm \ref{thm:classical-override}). Furthermore, we provide an example of a restriction category that has more than one override operator (Ex~\ref{ex:override-not-unique}), which was found using the help of Mace4~\cite{Prover9Mace4}. 

There are a number of promising directions for future work. For example,  a further operator on partial functions, also strongly related to the override operator, is the \emph{restricted union} $f \curlyvee g = (f \override g) \wedge (g \override f)$ where $\wedge$ indicates the meet of the partial functions in question~\cite{Stokes2024}. Restriction categories with meets are relatively well-studied~\cite{Guo2012,Cockett2012-2,Nester2024Thesis}, and we imagine that studying the restricted union operator in override restriction categories with meets would be a natural extension of the work presented here.

\textbf{Conventions:} We assume that the reader is familiar with the basics of category theory. For an arbitrary category $\mathbb{X}$, we will denote objects using capital letters $A$, $B$, $C$ etc., homsets will be denoted as $\mathbb{X}(A,B)$ and maps will be denoted by lowercase letters $f,g,h, etc. \in \mathbb{X}(A,B)$. Arbitrary maps will be denoted using an arrow $f: A \to B$, identity maps as $1_A: A \to A$, and for composition we will use \emph{diagrammatic} notation, that is, the composition of $f: A \to B$ followed by $g: B \to C$ is denoted as $fg: A \to C$. Throughout the paper, we will encounter various binary operations, such as $\override$, $\update$, $\vee$, etc. Composition is assumed to have a higher precedence than everything else, for example $fg \override hk = (fg) \override (hk)$.

\section{Restriction Categories}\label{sec:restriction}

In this background section, we briefly review the basics of restriction categories, some key concepts such as the canonical preorder and binary joins, which will play key roles for the story of this paper, as well as our main running examples. For a deeper introduction to restriction categories, see \cite{Cockett02,Cockett2023,Cockett2009Boolean}. 

\begin{definition}\cite[Sec 2.1.1]{Cockett02}\label{def:restcat} A \textbf{restriction category} is a category $\X$ equipped with a family of unary operations $\rest{-} : \X(A,B) \to \X(A,A)$ (for all pairs of objects $A$ and $B$) satisfying:
  \begin{enumerate}[{\bf [R.1]}]
  \item \label{R1} $\rest{f}f = f$ for all $f \in \X(A,B)$
  \item \label{R2} $\rest{f}\rest{g} = \rest{g}\rest{f}$ for all $f \in \X(A,B)$ and $g \in \X(A,C)$
  \item \label{R3} $\rest{\rest{f}g} = \rest{f}\rest{g}$ for all $f \in \X(A,B)$ and $g \in \X(A,C)$
  \item \label{R4} $f\rest{g} = \rest{fg}f$ for all $f \in \X(A,B)$ and $g \in \X(B,C)$
  \end{enumerate}
We call $\rest{f}$ the \textbf{restriction idempotent} of $f$. Furthermore, in a restriction category $\X$, for maps $f,g \in \mathbb{X}(A,B)$, we say that: 
\begin{enumerate}[{\em (i)}]
\item $f$ is \textbf{total} if $\rest{f} = 1_A$;
\item $g$ \textbf{extends} $f$, written $f \leq g$, if $\rest{f}g =f$;
\item $f$ and $g$ are \textbf{compatible}, written $f \smile g$ if $\rest{g}f = \rest{f}g$. 
   \end{enumerate}
\end{definition}

The main intuition for a restriction category is that maps $f$ are partial and the restriction $\rest{f}$ captures the domain of definition of $f$. So a map $f$ being total is interpreted as $f$ being everywhere defined. Moreover, the name restriction idempotent is justified since $\rest{f}$ is indeed an idempotent. Now $f \leq g$ is interpreted as saying that whenever $f$ is defined, $g$ is also defined and is equal to $f$, while $f \smile g$ means that whenever they are both defined, they are equal. In fact, $\leq$ is an actual partial order, making each homset a poset and every restriction category poset-enriched. As such, we can consider various poset related concepts of maps in restriction category such as, in particular, joins of maps. However, intuitively speaking, for the join of maps $f$ and $g$ to make sense, $f$ and $g$ need to be equal where they are defined, in other words, we can only take joins of compatible maps. 

\begin{definition}\label{def:join}\cite[Def 6.7 \& Def 10.1]{Cockett2009Boolean} A restriction category $\X$ is said to have \textbf{binary joins} in case for all $f,g \in \X(A,B)$ with $f \smile g$, there is a map $f \vee g \in \X(A,B)$ such that:
\begin{enumerate}[{\bf [J.1]}]
  \item $f \leq f \vee g$ and $g \leq f \vee g$, that is, $\rest{f}(f\vee g) = f$ and $\rest{g}(f\vee g) = g$;
  \item For any $h \in \X(A,B)$, $f \leq h$ and $g \leq h$, then $f \vee g \leq h$; 
  \item For any $k \in \X(A^\prime,A)$, $k(f \vee g) = kf \vee kg$. 
    \end{enumerate}
\end{definition}

So for $f \smile g$, we interpret $f \vee g$ as doing $f$ where $f$ is defined and not $g$, doing $g$ where $g$ is defined and not $f$, and since $f \smile g$, when they are both defined we can do either $f$ or $g$ since they are equal then (and of course being undefined where both are undefined). Now \textbf{[J.1]} and \textbf{[J.2]} says $f \vee g$ is indeed the join in the usual poset sense, while \textbf{[J.3]} tells us that joins are preserved by pre-composition. In fact, it follows that joins are also preserved by post-composition and the restriction idempotent of the join is the join of the restriction idempotents. Moreover, having binary joins means having joins of any non-empty finite families of compatible maps. We will discuss empty joins in Sec~\ref{sec:classical}. 

Here are now our main running examples of restriction categories (with binary joins). For a list of many other examples of restriction categories, see \cite[Sec 2.1.3]{Cockett02}.

\begin{example}\label{ex:par-rest-cat}
  The prototypical restriction category is the category $\mathsf{Par}$ of sets and partial functions, where for a partial function $f : X \to Y$, its restriction idempotent $\rest{f} : X \to X$ is defined as follows: $\rest{f}(x) =
    \begin{cases}
      x &\text{ if } f(x)\downarrow\\
      \uparrow &\text{ if } f(x)\uparrow
    \end{cases}$, where $\downarrow$ means defined and $\uparrow$ means undefined. The total maps correspond precisely to total functions in the usual sense. For partial functions $f: X \to Y$ and $g: X \to Y$, we have $f \leq g$ if $g(x)=f(x)$ whenever $f(x) \downarrow$, while $f \smile g$ if $f(x)=g(x)$ whenever both $f(x) \downarrow$ and $g(x) \downarrow$. Moreover, $\mathsf{Par}$ also has binary joins where for compatible partial functions $f: X \to Y$ and $g: X \to Y$, their join is the partial function $f \vee g: X \to Y$ defined as follows: 
\[ (f \vee g) (x) = \begin{cases} f(x) & \text{ if } f(x) \downarrow \text{ and } g(x) \uparrow \\
g(x) & \text{ if } f(x) \uparrow \text{ and } g(x) \downarrow \\
f(x)=g(x) & \text{ if } f(x) \downarrow \text{ and } g(x) \downarrow \\
\uparrow & \text{ if } f(x) \uparrow \text{ and } g(x) \uparrow  \end{cases} \]
\end{example}

\begin{example}\label{ex:trivial-rest-cat} Trivially, any category $\mathbb{X}$ is a restriction category where for every $f\in \mathbb{X}(A,B)$, $\rest{f} = 1_A$. Then every map is total, and moreover $f \leq g$ if and only if $f \smile g$ if and only if $f=g$. As such, we trivially have binary joins where $f \vee f = f$. 
\end{example}

\begin{example}\label{ex:cring-rest-cat} Let $\mathsf{Cring}_\bullet$ be the category of commutative rings and \textit{non-unital} ring morphisms, that is, functions $f: R\to S$ which preserve addition, $f(x+y) = f(x) + f(y)$, zeroes, $f(0) =0$, and multiplication, $f(xy) = f(x)f(y)$, but may not necessarily preserve the multiplicative unit, so $f(1)$ may not equal $1$. Then its opposite category $\mathsf{Cring}^{op}_\bullet$ is a restriction category, so $\mathsf{Cring}_\bullet$ is a \textit{co}restriction category, where for a non-unital ring morphism $f: R\to S$, its corestriction idempotent $\rest{f}: S \to S$ is defined as $\rest{f}(x) = f(1)x$. The total maps correspond precisely to the maps that do preserve the multiplicative unit, $f(1)=1$, in other words, actual ring morphisms. Moreover, $f \leq g$ if $f(1)g(x) = f(x)$, while $f \smile g$ if $g(1)f(x) = f(1) g(x)$. Moreover, $\mathsf{Cring}_\bullet$ also has binary joins where for compatible non-unital ring morphisms $f: R\to S$ and $g: R\to S$, their join is the non-unital ring morphism $f \vee g: R \to S$ defined as $(f \vee g)(x) = f(x) + g(x) - g(1)f(x) = f(x) + g(x) - f(1) g(x)$. 
\end{example}

\begin{example}\label{ex:meet-rest-cat} Every meet semi-lattice $(L, \wedge, \top)$ defines a restriction category with one object, whose maps are the elements of $L$, composition is given by the meet, $xy= x \wedge y$, the identity is the top element $\top$, and the every element is its own restriction idempotent, $\rest{x} = x$. Here, the only total map is the top element $\top$, everything is compatible, so $x \smile y$ for all $x,y \in L$, and the restriction category partial order corresponds to the usual semi-lattice partial order, that is, $x \leq y$ if and only if $x \wedge y = x$. If $(L, \wedge, \vee, \top)$ is an upper-bounded distributive lattice, then this gives a restriction category with binary joins which coincide with the lattice's joins. 
\end{example}

%More generally, the above example is an example of a \textbf{restriction monoid}, which can be defined as a one object restriction category or more explicitly as a monoid $M$ equipped with a unary operation $\rest{(-)}$ such that (1) $\rest{x} x = x$, (2) $\rest{x}~\rest{y} = \rest{y}~\rest{x}$, (3) $\rest{\rest{x}y} = \rest{x} ~\rest{y}$, and (4) $x\rest{y} = \rest{xy} x$, for all $x, y \in M$. 

We conclude this background section by recalling a number of useful basic identities we will require going forward:
\begin{lemma}[\cite{Cockett02,Cockett2023,Cockett2009Boolean}]\label{lem:restriction-basics}
  In a restriction category $\X$, we have:
  \begin{enumerate}[{\em (i)}]
  \item \label{lem:restriction-basics.idempotent}$\rest{f}~\rest{f} = \rest{f}$ for all $f \in \X(A,B)$
  %\item $\rest{fg} = \rest{f\rest{g}}$ for all $f \in \X(A,B)$ and $g \in \X(A,C)$
  \item \label{lem:restriction-basics.doublebar} $\rest{\rest{f}} = \rest{f}$ for all $f \in \X(A,B)$
  %\item $\rest{1_A} = 1_A$ for all objects $A$ of $\X$
  %\item \label{lem:restriction-basics.poset} For all objects $A$ and $B$, $\leq$ is a preorder on $\X(A,B)$
   \item \label{lem:restriction-basics.rest.leq.id} For all $f \in \X(A,A)$, $f \leq 1_A$ if and only if $f = \rest{f}$ .
 % \item \label{lem:restriction-basics.restf=restg} For all $f \in \X(A,B)$ and $g \in \X(A,C)$, if $\rest{f}g = g$ and $\rest{g}f = f$ then $\rest{f} = \rest{g}$.
  \item \label{lem:restriction-basics.leq.smile} For all $f, g \in \X(A,B)$, if $f \leq g$, then $f \smile g$. 
  \item \label{lem:restriction-basics.rest.smile} For all $f \in \X(A,B)$ and $g \in \X(A,C)$, $\rest{f} \smile \rest{g}$.
   \end{enumerate}
   Furthermore, if $\X$ also has binary joins, we have: 
   \begin{enumerate}[{\em (i)}]
   \setcounter{enumi}{5}
   \item \label{lem:restriction-basics.rest.join.rest} $\rest{f \vee g} = \rest{f} \vee \rest{g}$ for all $f,g \in \X(A,B)$ such that $f \smile g$.
   \item \label{lem:restriction-basics.rest.join.composition} $(f \vee g)h = fh \vee gh$ for all $f,g \in \X(A,B)$ such that $f \smile g$, and $h \in \X(B,C)$.
   \item \label{lem:restriction-basics.rest.join.assoc} $(f \vee g) \vee h = f \vee (g \vee h)$ for all $f,g, h \in \X(A,B)$ such that $f \smile g \smile h$.
   \item \label{lem:restriction-basics.rest.join.commutative} $f \vee g = g \vee f$ for all $f,g \in \X(A,B)$ such that $f \smile g$. 
   %\item \label{lem:restriction-basics.rest.join.idempotent.1} $f \vee f = f$ for all $f \in \mathbb{X}(A,B)$.
   \item \label{lem:restriction-basics.rest.join.idempotent.2} $\overline{g}f \vee f = f$ for all $f \in \mathbb{X}(A,B)$ and $g \in \mathbb{X}(A,C)$. 
   \end{enumerate}
\end{lemma}

\section{Override Restriction Categories}\label{sec:override}

In this section, we introduce the notion of an override operator for a restriction category, which is the main novel concept of this paper. We provide various examples (including an example which shows override operators are not unique) and also work out some basic identities. In particular, we show that every override restriction category has binary joins.  

\begin{definition}\label{def:override-rest-cat}
  An \textbf{override restriction category} is a restriction category $\mathbb{X}$ together with a family of binary operations $ -\override- : \X(A,B) \times \X(A,B) \to \X(A,B)$ (for all pairs of objects $A,B \in \mathbb{X}$) such that:
  \begin{enumerate}[{\bf [$\override$.1]}]
  \item \label{override.1} $(f \override g) \override h = f \override (g \override h)$ for all $f,g,h \in \X(A,B)$
  \item \label{override.2} $f \override \rest{f}g = f$ for all $f,g \in \X(A,B)$
  \item \label{override.3} $\rest{g}f \override f = f$ for all $f \in \X(A,B)$ and $g \in \X(A,C)$
  \item \label{override.4} $f(g \override h) = fg \override fh$ for all $f \in \X(A,B)$ and $g,h \in \X(B,C)$
  \item \label{override.5} $\rest{f \override g}\,h = \rest{f}h \override \rest{g}h$ for all $f,g \in \X(A,C)$ and $h \in \X(A,B)$
  \end{enumerate}
  We call $\override$ an \textbf{override operator} and $f \override g$ is called $f$ override $g$. 
\end{definition}

Intuitively, for maps $f$ and $g$, we think of $f \override g$ as doing $f$ when $f$ is defined and doing $g$ where $f$ is not defined (which is defined when $g$ is and undefined if not). {\bf [$\override$.1]} says that override operation is associative. {\bf [$\override$.2]} says that $f$ overriding $g$ where $f$ is defined is the same as just doing $f$. {\bf [$\override$.3]} says that the restricting $f$ to any restriction idempotent and overriding itself is the same as just doing $f$ again. {\bf [$\override$.4]} says that pre-composition preserves overriding. {\bf [$\override$.5]} says that $h$ restricted to the domain definition of $f$ overriding $h$ restricted to the domain of definition of $g$, is the same as $h$ restricted to the domain of definition of $f$ overriding $g$. As we will see below, from this it follows that the restriction of the override is equal to the override of the restrictions. %Here are now are main examples of override restriction categories. 

\begin{example}\label{ex:par-override} The restriction category $\mathsf{Par}$ (Example~\ref{ex:par-rest-cat}) is an override restriction category with override operator, where for partial functions $f: X \to Y$ and $g: X \to Y$, $f \override g: X \to Y$ is the partial function defined as follows: 
 \[(f \override g)(x) =
  \begin{cases}
    f(x) &\text{ if } f(x)\downarrow\\
    g(x) &\text{ if } f(x)\uparrow \text{ and } g(x)\downarrow\\
    \uparrow &\text{ otherwise}
  \end{cases}\] 
\end{example}

\begin{example}\label{ex:trivial-override}
  Every category seen as a trivial restriction category (Example~\ref{ex:trivial-rest-cat}) is an override restriction category with override operator given by $f \override g = f$.
\end{example}

\begin{example}\label{ex:cring-override}
The restriction category $\mathsf{CRING}^{op}_\bullet$ (Example~\ref{ex:par-rest-cat}) is an override restriction category, so $\mathsf{CRING}_\bullet$ has a co-override operator where for non-unital ring morphisms $f: R\to S$ and $g: R\to S$, $f \override g: R \to S$ is the non-unital ring morphism defined as $(f \override g)(x) =  f(x) + g(x) - f(1) g(x)$. 
\end{example}

\begin{example}\label{ex:lattice-override} An upper-bounded distributive lattice $(L,\wedge,\vee,\top)$, seen a restriction category as in Example~\ref{ex:meet-rest-cat}, is also an override restriction category whose override operation is given by the join, $x \override y = x \vee y$. 
\end{example}

That Examples~\ref{ex:par-override} and \ref{ex:cring-override} are override restriction categories will follow from Sec~\ref{sec:classical}, while Example~\ref{ex:trivial-override} and \ref{ex:lattice-override} are straightforward to check directly. Moreover, as we will see below, in many settings, override operators are unique. In fact, in all of the above examples, the override operator is unique. However in general, override operators need not be unique as exhibited by the following example:

\begin{example}\label{ex:override-not-unique} We define a one object restriction category whose only homset is the set $\{0,1,2,3,4,5\}$, where $0$ is the identity, the composition $*$ is defined via the Cayley table below on the left, and the domains of definition $\rest{x}$ are defined below on the right: 
  \begin{mathpar}
  \begin{tabular}{c | c c c c c c}
    $*$ & 0 & 1 & 2 & 3 & 4 & 5 \\
    \cline{1-7} 
    0 & 0 & 1 & 2 & 3 & 4 & 5 \\
    1 & 1 & 1 & 1 & 1 & 1 & 1 \\
    2 & 2 & 1 & 2 & 1 & 1 & 2 \\
    3 & 3 & 1 & 1 & 3 & 3 & 3 \\
    4 & 4 & 1 & 1 & 4 & 4 & 4 \\
    5 & 5 & 1 & 2 & 4 & 3 & 0
  \end{tabular}

  \overline{x} =
  \begin{cases}
    2 &\text{ if } x = 1 \text{ or } x = 2\\
    0 &\text{ otherwise}
  \end{cases}
\end{mathpar}
This restriction category admits two separate override operators $\override$ and $\blacktriangleright$, defined as follows: 
\begin{mathpar}
  \begin{tabular}{c | c c c c c c}
    $\override$ & 0 & 1 & 2 & 3 & 4 & 5 \\
    \cline{1-7} 
    0 & 0 & 0 & 0 & 0 & 0 & 0 \\
    1 & 4 & 1 & 1 & 3 & 4 & 3 \\
    2 & 0 & 2 & 2 & 5 & 0 & 5 \\
    3 & 3 & 3 & 3 & 3 & 3 & 3 \\
    4 & 4 & 4 & 4 & 4 & 4 & 4 \\
    5 & 5 & 5 & 5 & 5 & 5 & 5
  \end{tabular}

  \begin{tabular}{c | c c c c c c}
    $\blacktriangleright$ & 0 & 1 & 2 & 3 & 4 & 5 \\
    \cline{1-7} 
    0 & 0 & 0 & 0 & 0 & 0 & 0 \\
    1 & 3 & 1 & 1 & 3 & 4 & 4 \\
    2 & 0 & 2 & 2 & 0 & 5 & 5 \\
    3 & 3 & 3 & 3 & 3 & 3 & 3 \\
    4 & 4 & 4 & 4 & 4 & 4 & 4 \\
    5 & 5 & 5 & 5 & 5 & 5 & 5
  \end{tabular}
\end{mathpar}
This example was found using the program Mace4, which generates finite (counter-)examples for equational specifications~\cite{Prover9Mace4}. The corresponding code is given in Appendix~\ref{app:mace4-code}.
\end{example}

We now enumerate some interesting facts involving the override operator. 

\begin{lemma}\label{lem:override-facts}
  Let $\X$ be an override restriction category. Then for all $f,g,h \in \X(A,B)$, 
  %\vspace{-12pt}
  \begin{enumerate}[{\em (i)}]
  \begin{multicols}{2}
  \item \label{lem:override-facts.rest} $\rest{f \override g} = \rest{f} \override \rest{g}$;
  \item \label{lem:override-facts.idempotent} $f \override f = f$; 
  \item \label{lem:override-facts.total} If $f$ is total, then $f \override g = f$;
  \item \label{lem:override-facts.2} $f \leq f \override g$, that is, $\rest{f}(f \override g) = f$;
    \item \label{lem:override-facts.11} $\rest{f \override g} f = f$;
   \item \label{lem:override-facts.12} $\rest{f \override g} g  = g$; 
  \item \label{lem:override-facts.9} $f \override g = f$ if and only if $\rest{f} g = g$;
     \item \label{lem:override-facts.10} $(f \override g) \override f = f \override g$;
  \item \label{lem:override-facts.3} $g \leq f \override g$ if and only if $f \smile g$; 
  \columnbreak 
  \item \label{lem:override-facts.4} If $f \leq h$ and $g \leq h$, then $f \override g \leq h$;
\item \label{lem:override-facts.5} $f \leq g$ if and only if $f \override g = g$; 
 \item \label{lem:override-facts.6} $f \override g = g \override f$ if and only if $f \smile g$;
  \item \label{lem:override-facts.7} $\rest{f} \override \rest{g} = \rest{f \override g} = \rest{g \override f} = \rest{g} \override \rest{f}$;
  \item \label{lem:override-facts.8} $\rest{f} \override 1_A = 1_A = 1_A \override \rest{f}$; 
  \item $f$ is total if and only if $\rest{f} \override 1_A = \rest{f} = 1_A \override \rest{f}$.
  \end{multicols}
  \end{enumerate}
    %\vspace{-15pt}
\end{lemma}
\begin{proof} These are mostly straightforward to check. 
  \begin{enumerate}[{\em (i)}]
\item Using {\bf [$\override$.5]}, we have $\rest{f \override g} = \rest{f \override g}1_B \overset{\text{\tiny [$\override$.\ref{override.5}]}}{=}  \rest{f}1_B \override \rest{g}1_B = \rest{f} \override \rest{g}$.

\item Using {\bf [$\override$.3]}, we have $f \override f \overset{\text{\tiny [R.\ref{R1}]}}{=} \rest{f} f \override f \overset{\text{\tiny [$\override$.\ref{override.3}]}}{=} f$.

\item Suppose $f$ is total. Using {\bf [$\override$.2]}, we have $f \override g = f \override 1_A g \overset{\text{\tiny $f$ total}}{=} f \override \rest{f} g \overset{\text{\tiny [$\override$.\ref{override.2}]}}{=} f$. 

  \item We compute $\rest{f}(f \override g) \overset{\text{\tiny [$\override$.\ref{override.5}]}}{=}  \rest{f}f \override \rest{f}g \overset{\text{\tiny [R.\ref{R1}]}}{=} f \override \rest{f}g \overset{\text{\tiny [$\override$.\ref{override.2}]}}{=} f$. Thus $f \leq f \override g$. 

  \item We compute $\rest{f \override g} f \overset{\text{\tiny [$\override$.\ref{override.5}]}}{=} \rest{f}f \override \rest{g}f \overset{\text{\tiny [R.\ref{R1}]}}{=} f \override \rest{g} f \overset{\text{\tiny [$\override$.\ref{override.3}]}}{=} f$. 

  \item We compute $\rest{f \override g} g \overset{\text{\tiny [$\override$.\ref{override.5}]}}{=} \rest{f}g \override \rest{g} g \overset{\text{\tiny [R.\ref{R1}]}}{=} \rest{f}g \override g \overset{\text{\tiny [$\override$.\ref{override.3}]}}{=} g$. 

  \item Suppose $f \override g = f$. Then we have $\rest{f}g \overset{\text{\tiny Asmp.}}{=} \rest{f \override g} g \overset{\text{\tiny\ref{lem:override-facts}.(\ref{lem:override-facts.12}})}{=} g$. Conversely, suppose $\rest{f}g = g$. Then we compute $f \override g \overset{\text{\tiny Asmp.}}{=} f \override \rest{f} g \overset{\text{\tiny [$\override$.\ref{override.2}]}}{=} f$.

   \item Follows immediately from Lem \ref{lem:override-facts}.(\ref{lem:override-facts.11})+(\ref{lem:override-facts.9}). 
  
  \item If $g \leq f \override g$ then we have $\rest{f}g \overset{g \leq f \override g}{=} \rest{f}\rest{g}(f \override g) \overset{\text{\tiny [R.\ref{R2}]}}{=} \rest{g}\rest{f}(f \override g) \overset{\text{\tiny\ref{lem:override-facts}.(\ref{lem:override-facts.2}})}{=}   \rest{g}f$ and so $f \smile g$. Conversely, if $f \smile g$ then we have $\rest{g}(f \override g) \overset{\text{\tiny [$\override$.\ref{override.4}]}}{=}  \rest{g}f \override \rest{g}g \overset{\text{\tiny [R.\ref{R1}]}}{=} \rest{g}f \override g \overset{\text{\tiny $f \smile g$}}{=} \rest{f}g \override g \overset{\text{\tiny [$\override$.\ref{override.3}]}}{=} g$ and so $g \leq f \override g$.
  
  \item If $f \leq h$ and $g \leq h$ then we have $\rest{f \override g}h \overset{\text{\tiny [$\override$.\ref{override.5}]}}{=} \rest{f}h \override \rest{g}h \overset{f,g \leq h}{=} f \override g$, and so $f \override g \leq h$.

  \item Suppose $f \leq g$. On the one hand, by Lem \ref{lem:restriction-basics}.(\ref{lem:restriction-basics.leq.smile}), we have $f \smile g$, which then by Lem \ref{lem:override-facts}.(\ref{lem:override-facts.3}) gives us $g \leq f \override g$. On the other hand, since $f \leq f$, then Lem \ref{lem:override-facts}.(\ref{lem:override-facts.4}) gives us $f \override g \leq g$. Thus by antisymmetry of $\leq$, we get $f \override g = g$. Conversely, if $f \override g = g$, we have $\rest{f}g \overset{\text{\tiny Asmp.}}{=} \rest{f}(f \override g) \overset{\text{\tiny\ref{lem:override-facts}.(\ref{lem:override-facts.2}})}{=} f$, and so $f \leq g$. 
  
 % \item Immediate from parts 2,3, and 4.

\item Suppose $f \override g = g \override f$. By Lem \ref{lem:override-facts}.(\ref{lem:override-facts.3}) we have $g \leq g \override f$, but since $f \override g = g \override f$, this means $g \leq f \override g$, which by Lem \ref{lem:override-facts}.(\ref{lem:override-facts.3}) gives us $f \smile g$. Conversely, suppose $f \smile g$. Then by Lem \ref{lem:override-facts}.(\ref{lem:override-facts.2})+(\ref{lem:override-facts.3}) we have $g \leq g \override f$ and $f \leq g \override f$. Thus by Lem \ref{lem:override-facts}.(\ref{lem:override-facts.4}) this gives us that $f \override g \leq g \override f$. Similarly we can also argue that $g \override f \leq f \override f$. Thus by antisymmetry of $\leq$, we conclude that $f \override g = g \override f$. 

\item By Lem \ref{lem:restriction-basics}.(\ref{lem:restriction-basics.rest.smile}), we have that $\rest{f} \smile \rest{g}$. So the desired identity follows from Lem \ref{lem:override-facts}.(\ref{lem:override-facts.rest})+(\ref{lem:override-facts.6}). 

\item By \ref{lem:restriction-basics}.(\ref{lem:restriction-basics.rest.leq.id})+(\ref{lem:restriction-basics.rest.smile}), we know that $\rest{f} \leq 1_A$ and $\rest{f} \smile 1_A$. So the desired identity follows from Lem \ref{lem:override-facts}.(\ref{lem:override-facts.6})+(\ref{lem:override-facts.7}). 

\item This follows immediately from Lem \ref{lem:override-facts}.(\ref{lem:override-facts.8}). 
  \end{enumerate}
\end{proof}

A fundamental consequence of the above facts is that every override restriction category admits binary joins, where the join is given by the override. 

\begin{theorem}\label{thm:override-joins} If $\X$ is an override restriction category, then $\X$ has binary joins where for all $f,g \in \X(A,B)$, their join is $f \vee g = f \override g$. 
\end{theorem}
\begin{proof} We need to check that for compatible maps, $\override$ satisfies \textbf{[J.1]}, \textbf{[J.2]}, and \textbf{[J.3]}. Now \textbf{[J.1]} follows from Lem~\ref{lem:override-facts}.(\ref{lem:override-facts.2})+(\ref{lem:override-facts.3}), \textbf{[J.2]} follows from Lem~\ref{lem:override-facts}.(\ref{lem:override-facts.4}), and lastly \textbf{[J.3]} follows from {\bf [$\override$.4]}. So we conclude that an override restriction category has binary joins. 
\end{proof}

Since joins are unique, by \textbf{[J.2]}, it follows that having binary joins is a property of a restriction category rather than extra structure. From this, it follows that even if override operators are not necessarily unique, any two override operators must have the same domain of definition and be equal on compatible maps:

\begin{lemma}\label{lem:override-coincide}
  If $\override$ and $\blacktriangleright$ are both override operators on a restriction category $\mathbb{X}$, then for any maps $f,g \in \mathbb{X}(A,B)$, we have that $\rest{f \override g} = \rest{f \blacktriangleright g}$. Moreover, if $f \smile g$, then $f \override g = f \blacktriangleright g$. 
\end{lemma}
\begin{proof} Let us begin by explaining the second part. So suppose $f \smile g$. Then by Thm \ref{thm:override-joins}, $f \override g$ and $f \blacktriangleright g$ are both the join of $f$ and $g$. Since joins are unique it follows that $f \override g = f \blacktriangleright g$. Now drop the assumption $f$ and $g$ are compatible. By Lem \ref{lem:override-facts}.(\ref{lem:override-facts.rest}), we have that $\rest{f \override g} = \rest{f} \override \rest{g}$ and $\rest{f \blacktriangleright g} = \rest{f} \blacktriangleright \rest{g}$. However, recall from Lem \ref{lem:restriction-basics}.(\ref{lem:restriction-basics.leq.smile}) that we always have $\rest{f} \smile \rest{g}$. Therefore, by what we have already shown, we get that $\rest{f \override g} = \rest{f} \override \rest{g} = \rest{f} \blacktriangleright \rest{g} = \rest{f \blacktriangleright g}$ as desired. 
\end{proof}

We conclude this section with some results relating the behaviour of the override operator to the structure of the ambient restriction category, which will involve some degeneracy on the part of each. As such, we refer to these as \textit{collapse results}.

For our first collapse result, we find that the override operator is commutative if and only if every parallel map is compatible. A restriction category where all parallel maps are compatible is called a \textbf{restriction preorder} \cite[Def 3.3.1]{Giles2014}. Meet semi-lattices as in Example~\ref{ex:meet-rest-cat} are examples of restriction preorders. Moreover, it then follows that a restriction preorder has a necessarily unique override operator if and only if it has binary joins. 

\begin{lemma} Let $\X$ be restriction category. 
  \begin{enumerate}[{\em (i)}]
  \item \label{commutative.1} If $\X$ admits an override operator $\override$ which is commutative on each $\X(A,B)$, that is, $f \override g = g \override f$ for all $f,g \in \X(A,B)$, then $\X$ is a restriction preoder.%, that is, $f \smile g$ for all $f,g \in \X(A,B)$. 
   \item \label{commutative.2} If $\X$ is a restriction preorder with binary joins, then it has a unique override operator given by $f \override g = f \vee g$ for all $f,g \in \X(A,B)$. 
  \item \label{commutative.3} If $\X$ is also a restriction preorder, then $\X$ has an override operator if and only if $\X$ has binary joins.
  \end{enumerate}
\end{lemma}
\begin{proof} For (\ref{commutative.1}), this follows immediately from Lem \ref{lem:override-facts}.(\ref{lem:override-facts.6}). For (\ref{commutative.2}), suppose $\X$ is a restriction preorder with binary joins. Since every map is compatible, setting $f \override g = f \vee g$ is well-defined. It is easy to check that this is indeed an override operator using the basic identities of the join, so we leave this an exercise for the reader. Lastly (\ref{commutative.3}) follows from Thm \ref{thm:override-joins} and (\ref{commutative.2}). 
\end{proof}

Next, we consider the situation in which the override operator is simply a projection. We find that the first projection is an override operator precisely when the restriction structure of the ambient restriction category is trivial; while the second projection is an override operator precisely when the ambient category is a preorder, where recall that a category is preorder if that there is at most one map between two objects. Note that a restriction category which is also a preorder must be a trivial restriction category. 

\begin{lemma}Let $\X$ be restriction category. 
  \begin{enumerate}[{\em (i)}]
  \item \label{first.proj} The first projection $f \override g = f$ for all $f,g \in \X(A,B)$ is an override operator for $\X$ if and only if $\X$ is a trivial restriction category, that is, every map is total. 
  \item \label{second.proj} The second projection $f \override g = g$ for all $f,g \in \X(A,B)$ is an override operator for $\X$ if and only if $\X$ is a preorder.
  \end{enumerate}
\end{lemma}
\begin{proof} For (\ref{first.proj}), the $\Leftarrow$ direction is trivially checked. For the $\Rightarrow$ direction, suppose that $f \override g = f$ is an override operator. In particular this means we have $\rest{f} \override 1_A = \rest{f}$. However by Lem \ref{lem:override-facts}.(\ref{lem:override-facts.8}), we always have $\rest{f} \override 1_A = 1_A$. Therefore, $\rest{f} = 1_A$, so every map in $\X$ is total. For (\ref{second.proj}), the $\Leftarrow$ direction is trivially checked since there is at most one map between two objects. For the $\Rightarrow$ direction, suppose $f \override g = g$ is an override operator. In particular this means that $1_A \override \rest{f} = \rest{f}$. However by Lem \ref{lem:override-facts}.(\ref{lem:override-facts.8}), we always have $1_A \override \rest{f} = 1_A$. Therefore, $\rest{f} = 1_A$, so every map in $\X$ is total. Now let $f,g \in \X(A,B)$. Since $f$ is total, by Lem \ref{lem:override-facts}.(\ref{lem:override-facts.total}), this would mean that $f \override g = f$. However since by assumption $f \override g = g$, it follows that $f =g$. So $\X$ is a preorder. 
\end{proof}

\section{The Update Operator}\label{sec:update}

In semigroup theory, override operators are closely related to another kind of special operators called \textit{update} operators. In this section, we discuss the induced \textit{update operator} associated to an override operator, which is derived by pre-composing the override operator with the restriction idempotent of its second input. We will justify this definition by showing that the override and update operators satisfy Jackson and Stokes' axioms from~\cite[Sec 5]{Jackson2021}. 

\begin{definition}\label{def:update} In an override restriction $\X$, the \textbf{update operator} $\update$ is family of binary operations $ -\update- : \X(A,B) \times \X(A,B) \to \X(A,B)$ (for all pairs of objects $A,B \in \mathbb{X}$) defined as $f \update g = \rest{f}(g \override f)$.
\end{definition}

Now recall that $g \override f$ says do $g$ when $g$ is defined and do $f$ when $g$ is undefined. Then restricting $g \override f$ do the domain of definition of $f$ tells us that $f \update g$ should be interpreted as doing $g$ when $f$ is defined and doing $f$ when $g$ is undefined. Of course different override operators will result in different update operators. %Here are the induced update operators for our main examples of override restriction categories.

\begin{example}\label{ex:par-update} For the override restriction category $\mathsf{Par}$ (Example~\ref{ex:par-override}), for partial functions $f: X \to Y$ and $g: X \to Y$, $f \update g: X \to Y$ is the partial function defined as follows: 
 \[(f \update g)(x) =
  \begin{cases}
    f(x) &\text{ if } f(x)\downarrow \text{ and } g(x) \uparrow\\
    g(x) &\text{ if } f(x)\downarrow \text{ and } g(x)\downarrow\\
    \uparrow &\text{ otherwise}
  \end{cases}\] 
\end{example}

\begin{example}\label{ex:trivial-update} For a trivial restriction category (Example~\ref{ex:trivial-override}), $f \update g = g$. 
\end{example}

\begin{example}\label{ex:cring-update} For the override restriction category $\mathsf{CRING}^{op}_\bullet$ (Example~\ref{ex:cring-override}), in $\mathsf{CRING}_\bullet$, for non-unital ring morphisms $f: R\to S$ and $g: R\to S$, $f \update g: R \to S$ is the non-unital ring morphism defined as $(f \update g)(x) =  f(1)g(x) + f(x) - g(1) f(x)$. 
\end{example}

\begin{example}\label{ex:lattice-update} For an upper-bounded distributive lattice $(L,\wedge,\vee,\top,\bot)$, seen as an override restriction category as in Example~\ref{ex:lattice-override}, $x \update y = x \wedge (y \vee x)$. 
\end{example}

In~\cite[Thm 5.1 \& Rem 5.2]{Jackson2021}, Jackson and Stokes give axioms characterising the algebras over the signature $\{\override_{/2},\update_{/2}\}$ that are representable as systems of partial functions with the override and update operations given in Example~\ref{ex:par-override} and \ref{ex:par-update}. Treating these axioms as the equations of an algebraic theory, we obtain the following definition of an \textit{override-update algebra}. 

\begin{definition}An \textbf{override-update algebra} $(X,\override,\update)$ consists of a set $X$ together with binary operations $    -\override- : X \times X \to X$ and $-\update- : X \times X \to X$ such that the following equations are satisfied for all $x,y,z \in X$:
    \vspace{-10pt}
  \begin{enumerate}[{\bf [OU.1]}]
    \begin{multicols}{2}
  \item $(x \override y) \override z = x \override (y \override z)$
  \item $x \override x = x$
  \item $x = x \update (x \override y)$
  \columnbreak
  \item \label{OU4} $x \override y = (y \update x) \override x$
  \item $(x \update y) \update z = x \update (z \override y)$
  \item $x \update (y \update (x \update z)) = x \update ((y \update x) \update z)$
    \end{multicols}
  \end{enumerate}
      \vspace{-12pt}
\end{definition}

To justify the definition of the update operator in an override restriction category (and as a kind of sanity check), we show that the homsets of an override restriction category are override-update algebras. 

\begin{theorem}\label{thm:ou-algebra}
  Let $\X$ be an override restriction category. Then for all objects $A,B$ of $\X$, $(\X(A,B),\override,\update)$ is an override-update algebra. 
\end{theorem}
\begin{proof} \noindent {\bf [OU.1]} is precisely {\bf [$\override$.1]} while 

  \noindent {\bf [OU.2]} is precisely Lem~\ref{lem:override-facts}.(\ref{lem:override-facts.2}).

 \noindent {\bf [OU.3]} We compute $f \update (f \override g) \overset{\text{\tiny Def.}}{=} \rest{f}((f \override g) \override f) \overset{\text{\tiny\ref{lem:override-facts}.(\ref{lem:override-facts.10}})}{=} \rest{f}(f \override g) \overset{\text{\tiny\ref{lem:override-facts}.(\ref{lem:override-facts.2}})}{=}  f$. 

 \noindent {\bf [OU.4]} We compute $(g \update f) \override f \overset{\text{\tiny Def.}}{=} \overline{g}(f \override g) \override f \overset{\text{\tiny\ref{lem:override-facts}.(\ref{lem:override-facts.2}})}{=} \overline{g}(f \override g) \override \overline{f}(f \override g) \overset{\text{\tiny [$\override$.\ref{override.5}]}}{=}  \rest{g \override f}(f \override g) \\
    \overset{\text{\tiny\ref{lem:override-facts}.(\ref{lem:override-facts.7}})}{=}  \rest{f \override g}(f \override g) \overset{\text{\tiny [R.\ref{R1}]}}{=} f \override g$. 

      \noindent {\bf [OU.5]} We compute:
    \begin{gather*}
   (f \update g) \update h
        \overset{\text{\tiny Def.}}{=}  \rest{f}(g \override f) \update h
        \overset{\text{\tiny Def.}}{=}  \rest{\rest{f}(g \override f)}(h \override \rest{f}(g \override f))
        \overset{\text{\tiny [$\override$.\ref{override.5}]}}{=}  \rest{\rest{f}g \override \rest{f}f}(h \override \rest{f}(g \override f))
      \\ \overset{\text{\tiny [R.\ref{R1}]}}{=} \rest{\rest{f}g \override f}(h \override \rest{f}(g \override f))
      \overset{\text{\tiny\ref{lem:override-facts}.(\ref{lem:override-facts.7}})}{=}  \rest{f \override \rest{f}g}(h \override \rest{f}(g \override f))
       \overset{\text{\tiny [$\override$.\ref{override.2}]}}{=} \rest{f}(h \override \rest{f}(g \override f))
      \\\overset{\text{\tiny [$\override$.\ref{override.4}]}}{=} \rest{f}h \override \rest{f}\,\rest{f}(g \override f)
      \overset{\text{\tiny\ref{lem:restriction-basics}.(\ref{lem:restriction-basics.idempotent}})}{=} \rest{f}h \override \rest{f}(g \override f) \overset{\text{\tiny [$\override$.\ref{override.4}]}}{=} \rest{f}(h \override(g \override f))\\
       \overset{\text{\tiny [$\override$.\ref{override.1}]}}{=} \rest{f}((h \override g) \override f)
      \overset{\text{\tiny Def.}}{=} f \update (h \override g) 
    \end{gather*}
  
  \noindent {\bf [OU.6]} We compute: 
    \begin{gather*}
       f \update ((g \update f) \update h)
        \overset{\text{\tiny Def.}}{=}  f \update (\rest{g}(f \override g) \update h) \\
        \overset{\text{\tiny Def.}}{=}  f \update \rest{\rest{g}(f \override g)}(h \override \rest{g}(f \override g)) \overset{\text{\tiny Def.}}{=}  \rest{f}(\rest{\rest{g}(f \override g)}(h \override \rest{g}(f \override g)) \override f) \\
      \overset{\text{\tiny [$\override$.\ref{override.4}]}}{=} \rest{f}(\rest{\rest{g}f \override \rest{g}g}(h \override \rest{g}(f \override g)) \override f) \overset{\text{\tiny\ref{lem:override-facts}.(\ref{lem:override-facts.7}})}{=} \rest{f}(\rest{\rest{g}g \override \rest{g}f}(h \override \rest{g}(f \override g)) \override f)\\
      \overset{\text{\tiny [$\override$.\ref{override.4}]}}{=}  \rest{f}(\rest{\rest{g}(g \override f)}(h \override \rest{g}(f \override g)) \override f) \overset{\text{\tiny\ref{lem:override-facts}.(\ref{lem:override-facts.2}})}{=} \rest{f}(\rest{g}(h \override \rest{g}(f \override g)) \override f) \\
      \overset{\text{\tiny [$\override$.\ref{override.4}]}}{=} (\rest{f}\rest{g}h \override (\rest{f}\rest{g}\,\rest{g}f \override \rest{f}\rest{g}g)) \override \rest{f}f \overset{\text{\tiny\ref{lem:restriction-basics}.(\ref{lem:restriction-basics.idempotent}})}{=} (\rest{f}\,\rest{f}\rest{g}h \override (\rest{f}\rest{g}f \override \rest{f}\rest{g}g)) \override \rest{f}f\\
      \overset{\text{\tiny [R.\ref{R2}]}}{=} (\rest{f}\rest{g}\rest{f}h \override (\rest{f}\rest{g}f \override \rest{f}\rest{g}g)) \override \rest{f}f \overset{\text{\tiny [$\override$.\ref{override.1}]}}{=} ((\rest{f}\rest{g}\rest{f}h \override \rest{f}\rest{g}f) \override \rest{f}\rest{g}g) \override \rest{f}f  \overset{\text{\tiny [$\override$.\ref{override.4}]}}{=} \rest{f}(\rest{g}(\rest{f}(h \override f) \override g) \override f)\\
       \overset{\text{\tiny Def.}}{=} f \update \rest{g}(\rest{f}(h \override g) \override g)
        \overset{\text{\tiny Def.}}{=} f \update (h \update (\rest{f}(h \override f)))
        \overset{\text{\tiny Def.}}{=} f \update (g \update (f \update h))
    \end{gather*}
So we conclude that $(\X(A,B),\override,\update)$ is an override-update algebra. \end{proof}

It is natural to wonder about the converse, that is, whether or not having an override-update algebra structure on every hom-set $\X(A,B)$ is enough to make $\X$ into an override restriction category. We show that it is not, by means of a counterexample.

\begin{example} Let $\lbrace \top, \bot \rbrace$ be the two element lattice with top $\top$ and bottom element $\bot$. Now taking $x \override y = x \wedge y = x \update y$ yields an override-update algebra. However, this override operator does not define an override restriction category when we view $\lbrace \top, \bot \rbrace$ as a restriction category as in Example~\ref{ex:meet-rest-cat}: we have $\top \override \rest{\bot}\top = \bot \neq \top$, and so {\bf [$\override$.2]} fails.
\end{example}

We remark that this is to be expected: the axioms of an override restriction category concern the interaction of override with restriction and multiple hom-sets, while the axioms of an override-update algebra do not. So it is not surprising that the override restriction categories are stronger.

Here is now a list of interesting identities for the update operator. 

\begin{lemma}\label{lem:update-properties}
 Let $\X$ be an override restriction category. Then for all $f,g \in \X(A,B)$, 
     \vspace{-12pt}
  \begin{enumerate}[{\em (i)}]
      \begin{multicols}{2}
  \item \label{lem:update-properties.1} $\rest{f \update g} = \rest{f}$
  \item \label{lem:update-properties.2} $\rest{f} \update \rest{g} = \rest{f}$
  \item \label{lem:update-properties.3} $(f \update g) \update g = f \update g$
  \item \label{lem:update-properties.4}  $(f \update g) \update f = f$
  \item \label{lem:update-properties.idempotent} $f \update f = f$
    \columnbreak
  \item \label{lem:update-properties.6} $f \update (g \update f) = f$
  \item \label{lem:update-properties.7} $f \update (f \update g) = f \update g$
  \item \label{lem:update-properties.commutative} $f \update g = g \update f$ if and only if $f=g$
  \item \label{lem:update-properties.smile} $f \update g = f$ if and only if $f \smile g$
  \item \label{lem:update-properties.shelf} $(f \update g) \update h= (f \update h) \update (g \update h)$
  \end{multicols}
  \end{enumerate}
      \vspace{-15pt}
\end{lemma}
\begin{proof} These are mostly straightforward to check. 
  \begin{enumerate}[{\em (i)}]
  \item We compute $\rest{f \update g} \overset{\text{\tiny Def.}}{=} \rest{\rest{f}(g \override f)}    \overset{\text{\tiny [R.\ref{R3}]}}{=} \rest{f}\,\rest{g \override f} \overset{\text{\tiny\ref{lem:override-facts}.(\ref{lem:override-facts.7}})}{=}  \rest{f}\,\rest{f \override g} \overset{\text{\tiny [R.\ref{R3}]}}{=} \rest{\rest{f}(f \override g)} \overset{\text{\tiny\ref{lem:override-facts}.(\ref{lem:override-facts.2}})}{=}  \rest{f}$
  \item We compute $\rest{f} \update \rest{g}  \overset{\text{\tiny Def.}}{=} \rest{\rest{f}} (\rest{g} \override \rest{f})   \overset{\text{\tiny\ref{lem:restriction-basics}.(\ref{lem:restriction-basics.doublebar}})}{=}  \rest{f}(\rest{g} \override \rest{f}) \overset{\text{\tiny\ref{lem:override-facts}.(\ref{lem:override-facts.rest}})}{=}  \rest{f}\, \rest{g \override f} \overset{\text{\tiny [R.\ref{R3}]}}{=} \rest{\rest{f}(g \override f)}  \overset{\text{\tiny Def.}}{=}  \rest{f \update g}  \overset{\text{\tiny\ref{lem:update-properties}.(\ref{lem:update-properties.1}})}{=}  \rest{f} $
  \item We compute:
    \begin{gather*}
    (f \update g) \update g \overset{\text{\tiny Def.}}{=} \rest{f}(g \override f) \update g \overset{\text{\tiny Def.}}{=} \rest{\rest{f}(g \override f)}(g \override \rest{f}(g \override f)) \overset{\text{\tiny [R.\ref{R3}]}}{=} \rest{f}~\rest{g \override f}(g \override \rest{f}(g \override f)) \\
\overset{\text{\tiny\ref{lem:override-facts}.(\ref{lem:override-facts.7}})}{=}  \rest{f}~\rest{f \override g}(g \override \rest{f}(g \override f)) \overset{\text{\tiny [R.\ref{R3}]}}{=}  \rest{\rest{f}(f \override g)}(g \override \rest{f}(g \override f)) \overset{\text{\tiny\ref{lem:override-facts}.(\ref{lem:override-facts.2}})}{=}  \rest{f}(g \override \rest{f}(g \override f))\\
\overset{\text{\tiny [$\override$.\ref{override.4}]}}{=}  \rest{f}g \override \rest{f}\,\rest{f}(g \override f))  \overset{\text{\tiny\ref{lem:restriction-basics}.(\ref{lem:restriction-basics.idempotent}})}{=}  \rest{f}g \override \rest{f}(g \override f)) \overset{\text{\tiny [$\override$.\ref{override.4}]}}{=} \rest{f}g \override (\rest{f}g \override \rest{f}f) \\
\overset{\text{\tiny [$\override$.\ref{override.1}]}}{=} (\rest{f}g \override \rest{f}g) \override \rest{f}f  \overset{\text{\tiny\ref{lem:override-facts}.(\ref{lem:override-facts.idempotent}})}{=}  \rest{f}g \override \rest{f}f \overset{\text{\tiny [$\override$.\ref{override.4}]}}{=}  \rest{f}(g \override f)\overset{\text{\tiny Def.}}{=} f \update g
    \end{gather*}
  \item We compute: 
    \begin{gather*}
       (f \update g) \update f \overset{\text{\tiny Def.}}{=}  \rest{f}(g \override f) \update f \overset{\text{\tiny Def.}}{=}  \rest{\rest{f}(g \override f)}(f \override \rest{f}(g \override f)) \overset{\text{\tiny [R.\ref{R3}]}}{=} \rest{f}~\rest{g \override f}(f \override \rest{f}(g \override f))
      \\\overset{\text{\tiny\ref{lem:override-facts}.(\ref{lem:override-facts.7}})}{=} \rest{f}~\rest{f \override g}(f \override \rest{f}(g \override f)) \overset{\text{\tiny [R.\ref{R3}]}}{=} \rest{\rest{f}(f \override g)}(f \override \rest{f}(g \override f)) \overset{\text{\tiny\ref{lem:override-facts}.(\ref{lem:override-facts.2}})}{=}  \rest{f}(f \override \rest{f}(g \override f)) \\
      \overset{\text{\tiny [$\override$.\ref{override.4}]}}{=}  \rest{f}f \override \rest{f}\,\rest{f}(g \override f)  \overset{\text{\tiny\ref{lem:restriction-basics}.(\ref{lem:restriction-basics.idempotent}})}{=} \rest{f}f \override \rest{f}(g \override f)  \overset{\text{\tiny [$\override$.\ref{override.4}]}}{=} \rest{f}f \override (\rest{f}g \override \rest{f}f) \\
      \overset{\text{\tiny [R.\ref{R1}]}}{=} f \override (\rest{f}g \override f) \overset{\text{\tiny [$\override$.\ref{override.1}]}}{=} (f \override \rest{f}g) \override f   \overset{\text{\tiny [$\override$.\ref{override.2}]}}{=} f \override f  \overset{\text{\tiny\ref{lem:override-facts}.(\ref{lem:override-facts.idempotent}})}{=} f
    \end{gather*}
  \item We compute $f \update f \overset{\text{\tiny Def.}}{=} \rest{f}(f \override f) \overset{\text{\tiny\ref{lem:override-facts}.(\ref{lem:override-facts.idempotent}})}{=} \rest{f} f  \overset{\text{\tiny [R.\ref{R1}]}}{=} f$ 
  \item We compute $f \update (g \update f)  \overset{\text{\tiny Def.}}{=}  \rest{f}((g \update f) \override f) \overset{\text{\tiny [OU.4]}}{=} \rest{f}(f \override g) \overset{\text{\tiny\ref{lem:override-facts}.(\ref{lem:override-facts.2}})}{=} f$
  \item We compute $f \update (f \update g) \overset{\text{\tiny Def.}}{=} \rest{f}((f \update g) \override f) \overset{\text{\tiny Def.}}{=} \rest{f}(\rest{f}(g \override f) \override f)  \overset{\text{\tiny [R.\ref{R1}]}}{=} \rest{f}(\rest{f}(g \override f) \override \rest{f}f)  \\
      \overset{\text{\tiny [$\override$.\ref{override.4}]}}{=}  \rest{f} \, \rest{f} ((g \override f) \override f) \overset{\text{\tiny\ref{lem:restriction-basics}.(\ref{lem:restriction-basics.idempotent}})}{=} \rest{f} \, ((g \override f) \override f) \overset{\text{\tiny [$\override$.\ref{override.1}]}}{=} \rest{f} \, (g \override (f \override f)) \overset{\text{\tiny\ref{lem:override-facts}.(\ref{lem:override-facts.idempotent}})}{=} \rest{f}(g \override f) \overset{\text{\tiny Def.}}{=} f \update g$ 
  \item Suppose $f \update g = g \update f$. Then we compute: $f \overset{\text{\tiny\ref{lem:update-properties}.(\ref{lem:update-properties.6}})}{=}  f \update (g \update f) \overset{\text{\tiny Assum.}}{=} f \update (f \update g) \overset{\text{\tiny\ref{lem:update-properties}.(\ref{lem:update-properties.7}})}{=}  f \update g \overset{\text{\tiny\ref{lem:update-properties}.(\ref{lem:update-properties.3}})}{=} (f \update g) \update g
        \overset{\text{\tiny Assum.}}{=} (g \update f) \update g \overset{\text{\tiny\ref{lem:update-properties}.(\ref{lem:update-properties.4}})}{=} g$. So $f = g$. The converse direction is precisely Lem~\ref{lem:update-properties}.(\ref{lem:update-properties.idempotent}). 
\item By Lem \ref{lem:override-facts}.(\ref{lem:override-facts.3}) (and symmetry of $\smile$), we have that $f \leq g \override f$ if and only if $f \smile g$. In other words, $f \update g = \overline{f}(g \override f) = f$ if and only if $f \smile g$. 
\item We compute $(f \update g) \update h \overset{\text{\tiny [OU.5]}}{=} f \update (h \override g) \overset{\text{\tiny [OU.4]}}{=} f \update ((g \update h) \override h) \overset{\text{\tiny [OU.5]}}{=} (f \update h) \update (g \update h)$
  \end{enumerate}
\end{proof}

It is worth mentioning that Lem~\ref{lem:update-properties}.(\ref{lem:update-properties.shelf}) says that the update operator defines a \textbf{right-shelf}~\cite{Crans2004}. Shelves arise in, and are mostly studied in, the context of knot theory, where they model the third Redeimeister move. We find this connection intriguing and would like to investigate it further in future work. 

We conclude this section by providing some collapse results relating to properties of the update operator. We first observe that the update operator is commutative precisely when the ambient override restriction category is a preorder. 

\begin{lemma}
  Let $\X$ be an override restriction category. Its update operator $\update$ is commutative on each $\X(A,B)$, that is, $f \update g = g \update f$ for all $f,g \in \X(A,B)$, if and only if $\X$ is a preorder.
\end{lemma}
\begin{proof} This follows immediately from Lem~\ref{lem:update-properties}.(\ref{lem:update-properties.idempotent})+(\ref{lem:update-properties.commutative}). 
\end{proof}

We now show that the update operator is the first (resp. second) projection precisely if the ambient override restriction category is a restriction preorder (resp. trivial restriction). 

\begin{lemma}Let $\X$ be override restriction category. 
  \begin{enumerate}[{\em (i)}]
  \item \label{update.first.proj} Its update operator $\update$ is the first projection, that is, $f \update g= f$ for all $f,g \in \X(A,B)$ if and only if $\X$ is a restriction preorder. 
  \item \label{update.second.proj} Its update operator $\update$ is the second project, that is, $f \update g = g$ for all $f,g \in \X(A,B)$ is an override operator for $\X$ if and only if $\X$ is a trivial restriction category. 
  \end{enumerate}
\end{lemma}
\begin{proof} For (\ref{update.first.proj}), this follows immidietaly from Lem~\ref{lem:update-properties}.(\ref{lem:update-properties.smile}). For (\ref{update.second.proj}), first suppose that $f \update g = g$. In particular we have that $1_A \update \rest{f} = \rest{f}$. However since recall from that $1_A= \rest{1_A}$ and from Lem~\ref{lem:update-properties}.(\ref{lem:update-properties.2}) that $\rest{1_A} \update \rest{f} = \rest{1_A} = 1_A$, it follows that $\rest{f} = 1_A$. Therefore every map is total and thus $\X$ is a trivial restriction category. Conversely suppose that $\X$ is a trivial restriction category, so every map is total. Then by Lem \ref{lem:override-facts}.(\ref{lem:override-facts.total}), we have that $f \override g = f$, and therefore it is easy to see that $f \update g = g$. 
\end{proof}

%We remark that the situation regarding the associativity of the derived update operator is more complex that that regarding its commutativity. In sets and partial functions (Example~\ref{ex:par-override}) the derived update operator is not associative. However, in an override restiction category with trivial restriction structure (Example~\ref{ex:trivial-override}) we have $f \update g = g$, which is easily seen to be associative. We might then guess that the derived update operator is associative only when the restriction structure is trivial. Not so! In the situation of Example~\ref{ex:lattice-override} the restriction structure is nontrivial and we have $f \update g = f$, which is also associative. All we can say is that while the derived update operator need not be associative, it sometimes is.

\section{Starting from Update}\label{sec:starting-from-update}

Until now, we have taken the override operator as primitive and defined the update operator in terms of override. In this section we flip things around, taking the update operator as primitive and defining the override operator in terms of update. However, as we will explain below, in order to go from an update operator to an override operator, one requires binary joins. As such, we define an update restriction category as a restriction category with binary joins and an update operator. We will then show that every override restriction category is an update restriction category via the induced update operator, and conversely that every update restriction category is an override restriction category. Furthermore, these constructions are inverses of each other and thus the notion of an update restriction category is precisely the same as that of override restriction category. In doing so we both learn a bit about the update operator and provide some justification for the decision to take the override operator as primitive, in that it is much more economical than the alternative explored here.

We begin by explain how in an override restriction category, we can recapture the override operator from the update operator and the join. 

\begin{lemma}\label{lem:uptate-to-override} In an override restriction category $\X$, for any $f,g \in \X(A,B)$, we have that $(g \update f) \smile f$ and $f \override g = (g \update f) \vee f$. 
\end{lemma}
\begin{proof} We first compute that $\rest{f}(g \update f) \overset{\text{\tiny Def.}}{=} \rest{f}\rest{g}(f \override g) \overset{\text{\tiny [R.\ref{R3}]}}{=} \rest{g}\rest{f}(f \override g) \overset{\text{\tiny\ref{lem:override-facts}.(\ref{lem:override-facts.2}})}{=} \rest{g}f$. Therefore we have that $(g \update f) \smile f$. Then {\bf [OU.4]} tells us precisely that $f \override g \overset{\text{\tiny [OU.\ref{OU4}]}}{=} (g \update f) \override f \overset{\text{\tiny Def.}}{=} (g \update f) \vee f$. 
\end{proof}

Thus, it is possible to express the override operator in terms of the update operator and binary joins. The question is whether or not there are sensible axioms for the update operator that allow us to recover the override operator in this manner. We propose the following. 

\begin{definition}\label{def:update-rest-cat}
  An \emph{update restriction category} is a restriction category $\mathbb{X}$ that has binary joins together with a family of binary operations $ -\update- : \X(A,B) \times \X(A,B) \to \X(A,B)$ (for all pairs of objects $A,B \in \mathbb{X}$) such that:
  \begin{enumerate}[{\bf [$\update$.1]}]
  \item \label{update.1} $f(g \update h) = fg \update fh$ for all $f \in \X(A,B)$ and $g,h \in \X(B,C)$
  \item \label{update.2} $\rest{f \update g} = \rest{f}$
  \item \label{update.3} $\rest{f}g \update f = \rest{g}f$ for all $f,g \in \X(A,B)$
  \item \label{update.4} $\rest{g}h \update \rest{f}h = \rest{g}h$ for all $f \in \X(A,B)$, $g \in \X(A,C)$ and $h \in \X(A,D)$
  \item \label{update.5} $(((h \update g) \vee g) \update f) \vee f = (h \update ((g \update f) \vee f)) \vee ((g \update f) \vee f)$ for all $f,g,h \in \X(A,B)$
  \end{enumerate}
  We call $\override$ an \textbf{update operator} and $f \update g$ is called $f$ update $g$. 
\end{definition}

We acknowledge that, while axioms {\bf [$\update$.1-4]} are reasonably comprehensible, axiom {\bf [$\update$.5]} is monstrous. It is obtained from axiom {\bf [$\override$.\ref{override.1}]} by replacing $f \override g$ with $(g \update f) \vee f$. Ideally, we would like to be able to replace {\bf [$\update$.5]} with a number of simpler axioms that imply it, but as of yet we have been unable to find such axioms. 

We are now ready to show that any update restriction category is an override restriction category via the derived override operator:
\begin{proposition}\label{prop:update-to-override}
  Let $\X$ be an update restriction category with update operator $\update$. Then $\X$ is an override restriction, with the same underlying restriction structure and override operator $\override$ defined as $f \override g = (g \update f) \vee f$. Moreover, the induced update operator of this override operator (Def~\ref{def:update}) is precisely $\update$. 
\end{proposition}
\begin{proof} We first need to explain why this override operator is well-defined. So let $f,g \in \X(A,B)$. Then we compute that $\rest{f}(g \update f) \overset{\text{\tiny [$\update$.\ref{update.1}]}}{=} \rest{f}g \update \rest{f}f \overset{\text{\tiny [R.\ref{R1}]}}{=} \rest{f}g \update f \overset{\text{\tiny [$\update$.\ref{update.3}]}}{=} \rest{g}f \overset{\text{\tiny [$\update$.\ref{update.2}]}}{=} \rest{g \update f}f$. Therefore $(g \update f) \smile f$, so we may take their join, and thus our proposed override operator $f \override g = (g \update f) \vee f$ is well defined. Now we show that is satisfies the five override operator axioms. 
  \begin{enumerate}[{\bf [$\override$.1]}]
  \item $(f \override g) \override h \overset{\text{\tiny Def.}}{=} (h \update ((g \update f) \vee f)) \vee ((g \update f) \vee f)   \overset{\text{\tiny [$\update$.\ref{update.5}]}}{=} (((h \update g) \vee g) \update f) \vee f \overset{\text{\tiny Def.}}{=} f \override (g \override h)$. 
  \item  $f \override \rest{f}g \overset{\text{\tiny Def.}}{=} (\rest{f}g \update f) \vee f \overset{\text{\tiny [$\update$.\ref{update.3}]}}{=} \rest{g}f \vee f \overset{\text{\tiny\ref{lem:restriction-basics}.(\ref{lem:restriction-basics.rest.join.idempotent.2}})}{=}  f$. 
  \item $\rest{g}f \override f \overset{\text{\tiny Def.}}{=} (f \update \rest{g}f) \vee \rest{g}f \overset{\text{\tiny [R.\ref{R1}]}}{=}(\rest{f}f \update \rest{g}f) \vee \rest{g}f \overset{\text{\tiny [$\update$.\ref{update.4}]}}{=} \rest{f}f \vee \rest{g}f \overset{\text{\tiny [R.\ref{R1}]}}{=} f \vee \rest{g}f \overset{\text{\tiny\ref{lem:restriction-basics}.(\ref{lem:restriction-basics.rest.join.idempotent.2}})}{=}  f$. 
  \item $f(g \override h) \overset{\text{\tiny Def.}}{=} f((h \update g) \vee g) \overset{\text{\tiny [J.3]}}{=}   f(h \update g) \vee fg \overset{\text{\tiny [$\update$.\ref{update.1}]}}{=} (fh \update fg) \vee fg \overset{\text{\tiny Def.}}{=} fg \override fh$. 
  \item $\rest{f \override g}h \overset{\text{\tiny Def.}}{=} \rest{(g \update f) \vee f}h \overset{\text{\tiny\ref{lem:restriction-basics}.(\ref{lem:restriction-basics.rest.join.rest}})}{=} (\rest{g \update f} \vee \rest{f})h \overset{\text{\tiny [$\update$.\ref{update.2}]}}{=} (\rest{g} \vee \rest{f})h \overset{\text{\tiny\ref{lem:restriction-basics}.(\ref{lem:restriction-basics.rest.join.composition}})}{=} \rest{g}h \vee \rest{f}h \overset{\text{\tiny [$\update$.\ref{update.5}]}}{=} (\rest{g}h \update \rest{f}h) \vee \rest{f}h \overset{\text{\tiny Def.}}{=} \rest{f}h \override \rest{g}h$. 
  \end{enumerate}
Thus $\override$ is an override operator. Now let $\blackdiamond$ be the induced updated operator from $\override$ via Def~\ref{def:update}. Then we compute $f \blackdiamond g \overset{\text{\tiny Def.}}{=} \rest{f}(g \override f) \overset{\text{\tiny Def.}}{=} \rest{f}((f \update g) \vee g) \overset{\text{\tiny [$\update$.\ref{update.2}]}}{=} \rest{f \update g}((f \update g) \vee g) \overset{\text{\tiny [J.2]}}{=}  f \update g$. 
\end{proof}

Conversely, we now show that an override restriction category with its induced update operator is also an update restriction category and that the induced override operator from the above proposition is again the starting override operator. 

\begin{proposition}\label{prop:override-to-update}
  Let $\X$ be an override restriction category with override operator $\override$. Then $\X$ is an update restriction, with the same underlying restriction structure and update operator as defined in Def~\ref{def:update}. Moreover, the induced override operator of this update operator (Prop~\ref{prop:update-to-override}) is precisely $\override$. 
\end{proposition}
\begin{proof} We show that $\update$ satisfies the axioms of an update operator. 
  \begin{enumerate}[{\bf [$\diamond$.1]}]
  \item $f(g \update h) \overset{\text{\tiny Def.}}{=} f\rest{g}(h \override g) \overset{\text{\tiny [R.\ref{R4}]}}{=} \rest{fg}f(h \override g) \overset{\text{\tiny [$\override$.\ref{override.4}]}}{=}  \rest{fg}(fh \override fg) \overset{\text{\tiny Def.}}{=} fg \update fh$
       \item This is Lemma~\ref{lem:update-properties}.(\ref{lem:update-properties.1}) 
       \item $\rest{f}g \update f \overset{\text{\tiny Def.}}{=}  \rest{\rest{f}g}(f \override \rest{f}g) \overset{\text{\tiny [$\override$.\ref{override.2}]}}{=} \rest{\rest{f}g}f \overset{\text{\tiny [R.\ref{R3}]}}{=} \rest{f}~\rest{g}f \overset{\text{\tiny [R.\ref{R2}]}}{=} \rest{g}~\rest{f}f \overset{\text{\tiny [R.\ref{R1}]}}{=} \rest{g}f$. 
      \item $\rest{g}h \update \rest{f}h \overset{\text{\tiny Def.}}{=}  \rest{\rest{g}h}(\rest{f}h \override \rest{g}h) \overset{\text{\tiny [R.\ref{R3}]}}{=} \rest{g}~\rest{h}(\rest{f}h \override \rest{g}h) \overset{\text{\tiny [$\override$.\ref{override.4}]}}{=} \rest{g}~\rest{h}~\rest{f}h \override \rest{g}~\rest{h}~\rest{g}h \overset{\text{\tiny [R.\ref{R2}]}}{=} \rest{g}~\rest{f}~\rest{h}h \override \rest{g}~\rest{g}~\rest{h}h \overset{\text{\tiny [R.\ref{R1}]}}{=} \rest{f}~\rest{g}h \override \rest{g}~\rest{g}h \overset{\text{\tiny\ref{lem:restriction-basics}.(\ref{lem:restriction-basics.idempotent}})}{=} \rest{f}~\rest{g}h \override \rest{g}h \overset{\text{\tiny Def.}}{=} \rest{f}~\rest{g}h \vee\rest{g}h \overset{\text{\tiny\ref{lem:restriction-basics}.(\ref{lem:restriction-basics.rest.join.idempotent.2}})}{=} \rest{g}h$. 
  \item $(((h \update g) \vee g) \update f) \vee f \overset{\text{\tiny Lem \ref{lem:uptate-to-override}}}{=} f \override (g \override h) \overset{\text{\tiny [$\override$.\ref{override.1}]}}{=} (f \override g) \override h \overset{\text{\tiny Lem \ref{lem:uptate-to-override}}}{=}  (h \update ((g \update f) \vee f)) \vee ((g \update f) \vee f)$. 
  \end{enumerate}
Thus $\update$ is an update operator. 

Now let $\blacktriangleright$ be the induced override operator from $\update$ via Prop~\ref{prop:update-to-override}. Then we compute  $f \blacktriangleright g \overset{\text{\tiny Def.}}{=}(g \update f) \vee f \overset{\text{\tiny Def.}}{=} \rest{g}(f \override g) \vee f \overset{\text{\tiny Def.}}{=}  \rest{g}(f \override g) \override f \overset{\text{\tiny\ref{lem:override-facts}.(\ref{lem:override-facts.2}})}{=}\rest{g}(f \override g) \override \rest{f}(f \override g) \overset{\text{\tiny [$\override$.\ref{override.5}]}}{=} \rest{g \override f}(f \override g) \overset{\text{\tiny\ref{lem:override-facts}.(\ref{lem:override-facts.7}})}{=}  \rest{f \override g}(f \override g) \overset{\text{\tiny [R.\ref{R1}]}}{=}  f \override g$. 
\end{proof}

Therefore, the constructions of Prop~\ref{prop:update-to-override} and Prop~\ref{prop:override-to-update} are inverses of each other, and thus we can conclude that: 

\begin{theorem}\label{thm:override=update} An override restriction category is precisely an update restriction category. 
\end{theorem}

\section{Override for Classical Restriction Categories}\label{sec:classical}

In this section we show that \textit{classical} restriction categories are override restriction categories, and in fact, they have a unique override operator. Briefly, a classical restriction category is a restriction with binary joins that also has \textit{restriction zero maps} and \textit{complements}, which allows for Boolean classical reasoning within a restriction category framework. For an in-depth introduction to classical restriction categories, we refer the reader to \cite{Cockett2023,Cockett2009Boolean}. 

We begin by reviewing restriction zero maps, which are zero maps in the usual sense whose restriction idempotent is also a zero map. 

\begin{definition}\cite[Sec 2.2]{Coc07} A restriction category $\X$ is said to have \textbf{restriction zero maps} if for all objects $A$ and $B$ of $\X$ there is a morphism $0_{A,B} \in \X(A,B)$ such that $\rest{0_{A,B}} = 0_{A,A}$ and for all objects $A,B,C,D$ of $\X$ and all morphisms $f \in \X(B,C)$ we have $0_{A,B}f = 0_{A,C}$ and $f0_{C,D} = 0_{B,D}$. As shorthand, we denote zero maps simply as $0$. 
\end{definition}

Intuitively, restriction zero maps are maps which are nowhere defined. We first observe that in an override restriction category, restriction zero maps behave well with respect to the override and update operators. 

\begin{lemma}\label{lem:override.zero}
  Let $\X$ be an override restriction category with restriction zero maps. Then for all $f \in \X(A,B)$, we have that: 
  \begin{enumerate}[{\em (i)}]
  \item \label{lem:override.zero.1} $0 \override f = f =f \override 0$, that is, $0$ is a unit for $\override$.
  \item \label{lem:override.zero.2} $0 \update f = 0$ and $f \update 0 = f$, that is, $0$ is a left-annihilator and right-unit for $\update$.
  \end{enumerate}
\end{lemma}
\begin{proof} These are straightforward to check, so we leave them as an exercise for the reader. 
\end{proof}

Before we review relative complements, it is useful to recall the notion of being disjoint. 

\begin{definition} \cite[Prop 6.2]{Cockett2009Boolean} In a restriction category $\mathbb{X}$ with restriction zero maps, for maps $f,g \in \X(A,B)$, we say that $f$ and $g$ are \textbf{disjoint}, written $f \perp g$, if $\overline{f}g = 0$ (or equivalently $\overline{g}f =0)$. 
\end{definition}

Intuitively, $f$ and $g$ are disjoint if where one is defined the other is not, or in other words, the domains of definition of $f$ and $g$ do not overlap. %With this, we are now ready to recall the definition of classical restriction categories. 

\begin{definition}\cite[Sec 13]{Cockett2009Boolean}\cite[Def 4.8]{Cockett2023}\label{def:classical-rest-cat}
  A \emph{classical restriction category} is a restriction category $\X$ which has binary joins and restriction zeroes such that for maps $f,g \in \X(A,B)$ with $f \leq g$, there exists a (necessarily) unique map $g \backslash f \in \X(A,B)$ such that $g \backslash f \perp f$ and $(g \perp f) \vee f = g$. We call $g \backslash f$ the \textbf{relative complement} of $f$ in $g$. Moreover, for any map $f \in \X(A,B)$ the \textbf{complement} of its restriction idempotent $\rest{f}$ is the relative complement of $\rest{f}$ in the identity $1_A$, that is, the map $\rest{f}^c \in \mathbb{X}(A,A)$ defined as $\rest{f}^c= 1_A\backslash \rest{f}$. 
\end{definition}

Intuitively, we interpret $g \backslash f$ as doing $g$ when $f$ is undefined and being undefined where $f$ is defined. Now recall that by Lemma~\ref{lem:restriction-basics}.(\ref{lem:restriction-basics.rest.leq.id}), for restriction idempotents $\rest{f}$ we always have that $\rest{f} \leq 1_A$, so taking its complement $\rest{f}^c$ is well-defined, which we interpret as capturing where $f$ is not defined.  

\begin{example} The restriction category $\mathsf{Par}$ (Example~\ref{ex:par-rest-cat}) is a classical restriction category where the restriction zero is the partial function $0_{X,Y}: X \to Y$ which is nowhere defined, while for partial functions $f: X \to Y$ and $g: X \to Y$ such that $f \leq g$, the relative complement is the partial function $g \backslash f: X \to Y$ is defined as follows: $(g \backslash f) (x) = \begin{cases} g(x) & \text{ if } f(x) \uparrow \text{ and } g(x) \downarrow \\
\uparrow & \text{ if } f(x) \downarrow \text{ or } g(x) \uparrow  \end{cases}$. In particular, the complement of the restriction idempotent $\rest{f}: X \to X$ is the partial function $\rest{f}^c: X \to X$ defined as follows: $\rest{f}^c(x) = \begin{cases} x & \text{ if } f(x) \uparrow  \\
\uparrow & \text{ if } f(x) \downarrow   \end{cases}$. 
\end{example}

\begin{example} The restriction category $\mathsf{Cring}^{op}_\bullet$ (Example~\ref{ex:cring-rest-cat}) is a classical restriction category, so $\mathsf{Cring}_\bullet$ is a coclassical corestriction category, where the corestriction zero is simply the zero morphism $0_{R,S}: R \to S$, while for non-unital ring morphisms $f: R \to S$ and $g: R\to S$ such that $f \leq g$, the relative complement is the non-unital ring morphism $g \backslash f: R\to S$ defined as follows $(g \backslash f) (x) = g(x) - f(1)g(x)$. In particular, the complement of the corestriction idempotent $\rest{f}: R\to R$ is the non-unital ring morphism $\rest{f}^c: R \to R$ defined as follows $\rest{f}^c(x) = x-f(1)x$. 
\end{example}

\begin{example} A Boolean algebra $(B, \wedge, \vee, \top, \bot, \neg)$, seen as a restriction category (Example~\ref{ex:meet-rest-cat}), is a classical restriction category where the restriction zero is the bottom element $0 = \bot$, and where for $x,y \in B$, the relative complement is $y \backslash x = \neg x \wedge y$. Thus since recall that $\rest{x} = x$ for all $x \in B$, its complement is just negation, $\rest{x}^c = \neg x$. 
\end{example}

We will now show that every classical restriction category admits a unique override operator capturing precisely the intuition that $f \override g$ is given by doing $f$ where $f$ is defined and doing $g$ where $f$ is not defined, which we can do using joins and complements. Here is first some basic identities that will be useful to prove our theorem below. 

\begin{lemma}[\cite{Cockett2023,Cockett2009Boolean}]\label{lem:classical} In a restriction category $\X$ with binary joins and restriction zeroes, we have: 
  \begin{enumerate}[{\em (i)}]
  \item \label{lem:classical.perp.smile}For maps $f,g \in \X(A,B)$, if $f \perp g$ then $f \smile g$.
  \item \label{lem:classical.zero.join} For all maps $f \in \X(A,B)$, $0 \smile f$ and $0 \vee f = f = f \vee 0$. 
    \end{enumerate}
Moreover, if $\X$ is a classical restriction category, then we have: 
  \begin{enumerate}[{\em (i)}]
  \setcounter{enumi}{2}
%  \item $\rest{\rest{f}^c} = \rest{f}^c$ for all $f \in \X(A,B)$.
  \item \label{lem:classical.comp.perp}$\rest{f}~\rest{f}^c = 0 = \rest{f}^c~\rest{f}$ for all $f \in \X(A,B)$. 
  \item \label{lem:classical.comp.join} $\rest{f} \vee \rest{f}^c = 1_A$ for all $f \in \X(A,B)$. 
    \item \label{lem:classical.R1} $\rest{f}^c f = 0$ for all $f \in \X(A,B)$. 
       \item \label{lem:classical.R2} $\rest{f}^c \rest{g} = \rest{g} \rest{f}^c$ for all $f \in \X(A,B)$ and $g \in \X(A,C)$. 
  \item \label{lem:classical.R3} $\rest{\rest{f}g}^c = \rest{f}^c \vee \rest{g}^c$ and $\rest{\rest{f}^c g}^c = \rest{f} \vee \rest{g}^c$ for all $f \in \X(A,B)$ and $g \in \X(A,C)$. 
\item \label{lem:classical.R4} $f \rest{g}^c = \rest{fg}^c f$ for all $f \in \X(A,B)$ and $g\in \X(B,C)$. 
\item \label{lem:classical.comp.rest.join} $\rest{f \vee g}^c = \rest{f}^c~\rest{g}^c$ for all $f,g \in \X(A,B)$ with $f \smile g$. 
  \end{enumerate}
\end{lemma}

\begin{theorem}\label{thm:classical-override}
  A classical restriction category $\X$ has a unique override operator $\override$ defined as $f \override g := f \vee \overline{f}^c g$. Moreover, its induced update operator is $f \update g = \overline{g}^c f \vee \rest{f}g $. 
\end{theorem}
\begin{proof} We first need to explain why $\override$ is well-defined. However it follows from Lem~\ref{lem:classical}.(\ref{lem:classical.comp.perp}) that clearly $f \perp \overline{f}^c g$, so by Lem~\ref{lem:classical}.(\ref{lem:classical.perp.smile}) we have $f \smile \overline{f}^c g$, and thus we can take their join. So $f \override g := f \vee \overline{f}^c g$ is well-defined. We now show that this satisfies the five axioms of an override operator. 
\begin{enumerate}[{\bf [$\override$.1]}]
\item We compute: 
%\vspace{-18pt}
\begin{gather*}
   (f \override g) \override h  \overset{\text{\tiny Def.}}{=} (f \override g) \vee \rest{f \override g}^c~h \overset{\text{\tiny Def.}}{=} \left( f \vee \overline{f}^c g \right) \vee \rest{f \vee \overline{f}^c g}^c~h  \\
   \overset{\text{\tiny\ref{lem:classical}.(\ref{lem:classical.comp.rest.join}})}{=} \left( f \vee \overline{f}^c g \right) \vee \rest{f}^c~ \rest{\overline{f}^c g}^c h \overset{\text{\tiny\ref{lem:classical}.(\ref{lem:classical.R3}})}{=}  \left( f \vee \overline{f}^c g \right) \vee \rest{f}^c (\rest{f} \vee \rest{g}^c) h \\ \overset{\textbf{[J.3]}}{=} \left( f \vee \overline{f}^c g \right) \vee \left( \rest{f}^c \rest{f} h \vee \rest{f}^c\rest{g}^c h \right) \overset{\text{\tiny\ref{lem:classical}.(\ref{lem:classical.R1}})}{=} \left( f \vee \overline{f}^c g \right) \vee \left( 0 h \vee \rest{f}^c\rest{g}^c h \right) \\ \overset{zero}{=} 
   \left( f \vee \overline{f}^c g \right) \vee \left( 0 \vee \rest{f}^c\rest{g}^c h \right) \overset{\text{\tiny\ref{lem:classical}.(\ref{lem:classical.zero.join}})}{=} \left( f \vee \overline{f}^c g \right) \vee \rest{f}^c\rest{g}^c h \overset{\text{\tiny\ref{lem:restriction-basics}.(\ref{lem:restriction-basics.rest.join.assoc}})}{=} f \vee \left( \overline{f}^c g \vee \rest{f}^c\rest{g}^c h \right)\\ \overset{\textbf{[J.3]}}{=} 
   f \vee \overline{f}^c \left( g \vee \rest{g}^c h \right) \overset{\text{\tiny Def.}}{=} f \vee \overline{f}^c (g \override h)  \overset{\text{\tiny Def.}}{=} f \override (g \override h)
\end{gather*}
\item We compute $\rest{g}f \override f \overset{\text{\tiny Def.}}{=} \rest{g}f \vee \rest{\rest{g}f}^cf\overset{\text{\tiny\ref{lem:classical}.(\ref{lem:classical.R3}})}{=}\rest{g}f \vee \left( \rest{g}^c \vee \rest{f}^c \right) f \overset{\text{\tiny\ref{lem:restriction-basics}.(\ref{lem:restriction-basics.rest.join.composition}})}{=} \left(\rest{g} \vee \left(\rest{g}^c \vee \rest{f}^c \right)\right) f \overset{\text{\tiny\ref{lem:restriction-basics}.(\ref{lem:restriction-basics.rest.join.assoc}})}{=} \left( \left( \rest{g} \vee \rest{g}^c \right) \vee \rest{f}^c \right) f \overset{\text{\tiny\ref{lem:classical}.(\ref{lem:classical.comp.join}})}{=} \left( 1_A \vee \rest{f}^c \right) f \overset{\text{\tiny\ref{lem:restriction-basics}.(\ref{lem:restriction-basics.rest.join.idempotent.2}})}{=}  1_A f = f$. 
\item We compute $f(g \override h) \overset{\text{\tiny Def.}}{=} f( g \vee \overline{g}^c h) \overset{\text{\tiny \bf[J.3]}}{=} f g \vee f\overline{g}^c h \overset{\text{\tiny\ref{lem:classical}.(\ref{lem:classical.R4}})}{=} f g \vee \overline{fg}^c fh  \overset{\text{\tiny Def.}}{=}  fg \override fh$. 
\item We compute:
%\vspace{-18pt}
\begin{gather*}
    \rest{f \override g}~h \overset{\text{\tiny Def.}}{=} \rest{f \vee \overline{f}^c g}~h \overset{\text{\tiny\ref{lem:restriction-basics}.(\ref{lem:restriction-basics.rest.join.rest}})}{=} \left( \rest{f} \vee \rest{\overline{f}^c g} \right)h \overset{\text{\tiny\ref{lem:restriction-basics}.(\ref{lem:restriction-basics.rest.join.composition}})}{=} \rest{f}h \vee \rest{\overline{f}^c g}h \\
    \overset{\text{\tiny\ref{lem:classical}.(\ref{lem:classical.R3}})}{=} \rest{f}h \vee \overline{f}^c \rest{g} h  \overset{\text{\tiny\ref{lem:classical}.(\ref{lem:classical.zero.join}})}{=} \left (\rest{f}h \vee \overline{f}^c \rest{g} h \right) \vee 0 
    \overset{zero}{=} \left( \rest{f}h \vee \overline{f}^c \rest{g} h \right) \vee \rest{g}0 \\ \overset{\text{\tiny\ref{lem:classical}.(\ref{lem:classical.R1}})}{=} \left( \rest{f}h \vee \overline{f}^c \rest{g} h \right) \vee \rest{g}\rest{h}^c h \overset{\text{\tiny\ref{lem:classical}.(\ref{lem:classical.R2}})}{=}  \left( \rest{f}h \vee \rest{f}^c\rest{g}h \right) \vee \rest{h}^c\rest{g}h \\ \overset{\text{\tiny\ref{lem:restriction-basics}.(\ref{lem:restriction-basics.rest.join.assoc}})}{=} \rest{f}h \vee \left( \rest{f}^c\rest{g}h  \vee \rest{h}^c\rest{g}h \right) \overset{\text{\tiny\ref{lem:restriction-basics}.(\ref{lem:restriction-basics.rest.join.composition}})}{=}  \rest{f}h \vee \left( \rest{f}^c \vee \rest{h}^c \right) \rest{g}h \overset{\text{\tiny\ref{lem:classical}.(\ref{lem:classical.R3}})}{=} \rest{f}h \vee \rest{\rest{f}h}^c \rest{g}h \overset{\text{\tiny Def.}}{=} \rest{f}h \override \rest{g}h
\end{gather*}
\end{enumerate}
So we conclude that $\override$ is indeed an override operator. Now for uniqueness, suppose we have another override operator $\blacktriangleright$. Lem~\ref{lem:override-facts}.(\ref{lem:override-facts.2}) gives that $f \leq f \blacktriangleright g$, and we also compute that: 
\begin{gather*}
    \rest{\rest{f}^c g}(f \blacktriangleright g)  \overset{\text{\tiny [$\override$.\ref{override.1}]}}{=} \rest{\rest{f}^c g}f \blacktriangleright \rest{\rest{f}^c g}~g \overset{\text{\tiny [R.\ref{R3}]}}{=} \rest{f}^c\rest{g}f \blacktriangleright \rest{f}^c\rest{g}g \overset{\text{\tiny\ref{lem:classical}.(\ref{lem:classical.R2}})}{=}\rest{g}~\rest{f}^cf \blacktriangleright \rest{f}^c\rest{g}g\\
    \overset{\text{\tiny\ref{lem:classical}.(\ref{lem:classical.R1}})}{=} 0 \blacktriangleright \rest{f}^c\rest{g}g \overset{\text{\tiny\ref{lem:override.zero}.(\ref{lem:override.zero.1}})}{=} \rest{f}^c\rest{g}g \overset{\text{\tiny [R.\ref{R1}]}}{=} \rest{\rest{f}^c} g
\end{gather*}
Thus $\rest{f}^c g \leq f \blacktriangleright g$. As such by \textbf{[J.2]}, we have that $f \override g = f \vee \rest{f}^c g \leq f \blacktriangleright g$, or in other words, $\overline{f \override g}(f \blacktriangleright g)$. However, recall from Lem~\ref{lem:override-coincide}, we always have that $\rest{f \override g}= \rest{f \blacktriangleright g}$ as well. Therefore by \textbf{[R.\ref{R1}]}, we get that $f \blacktriangleright g = f \override g$ as desired. Lastly, we compute the induced update operator as follows: 
\[f \diamond g \overset{\text{\tiny Def.}}{=} \rest{f}(g \override f) \overset{\text{\tiny Def.}}{=} \rest{f}(g \vee \overline{g}^c f) \overset{\text{\tiny [J.3]}}{=} \rest{f}g \vee \rest{f}\overline{g}^c f \overset{\text{\tiny\ref{lem:classical}.(\ref{lem:classical.R2}})}{=} \rest{f}g \vee \overline{g}^c\rest{f} f \overset{\text{\tiny [R.\ref{R1}]}}{=} \rest{f}g \vee \overline{g}^c f \overset{\text{\tiny\ref{lem:restriction-basics}.(\ref{lem:restriction-basics.rest.join.commutative}})}{=}  \overline{g}^c f \vee \rest{f}g\] 
So we indeed have that $f \update g = \overline{g}^c f \vee \rest{f}g$. 
\end{proof}

%\section{Future Work}
%We have introduced override restriction categories as a means to study the override and update operators in a restriction category framework, and have obtained a number of results concerning them. Our results both develop the theory of the override and update operators, and connect them to the wider body of work concerning restriction categories. We have, moreover, shown that our approach is compatible with the axioms for override and update given by Jackson and Stokes (Theorem~\ref{thm:ou-algebra}), which means that our work is in harmony with existing approaches to the study of the override and update operators. 

%Along similar lines, one might study the \emph{maximal iterate} operator~\cite{Jackson2011} in the context of override restriction categories. We imagine this to be closely connected to the notion of \emph{itegory}, which has recently been developed in order to study notions of iteration in restriction categories~\cite{Lemay2025}.
%First, some aspects of our development are unsatisfying: it would be nice to be able replace axiom {\bf [$\update$.5]} of update restriction categories (Definition~\ref{def:update-rest-cat}) with something more reasonable, and to have a better understanding of restriction categories that admit multiple override operators (Example~\ref{ex:override-not-unique}). 
%Also missing is an example of a restriction category with binary joins that does not admit an override operator, although it seems certain that such categories must exist.

%\newpage

\bibliographystyle{plainurl}% the mandatory bibstyle
\bibliography{citations}

@article{Berendsen2010,
  title={The axiomatization of override and update},
  author={Berendsen, J. and Jansen, D. N. and Schmaltz, J. and Vaandrager, F. W.},
  journal={Journal of Applied Logic},
  volume={8},
  number={1},
  pages={141--150},
  year={2010},
  publisher={Elsevier}
}

@article{Jackson2021,
  title={Override and Update},
  author={Jackson, M. and Stokes, T.},
  journal={Journal of Pure and Applied Algebra},
  volume={225},
  number={3},
  pages={106532},
  year={2021},
  publisher={Elsevier}
}

@article{Cvetko2013,
  title={Skew lattices and binary operations on functions},
  author={Cvetko-Vah, K. and Leech, J. and Spinks, M.},
  journal={Journal of Applied Logic},
  volume={11},
  number={3},
  pages={253--265},
  year={2013},
  publisher={Elsevier}
}

@book{bakker1980mathematical,
  title={Mathematical theory of program correctness},
  author={{de} Bakker, J.},
  year={1980},
  publisher={Prentice-Hall, Inc.}
}

@book{Spivey1992,
  title={The Z notation},
  author={Spivey, J. M. and Abrial, J.-R.},
  volume={29},
  year={1992},
  publisher={Prentice Hall Hemel Hempstead}
}

@article{Cockett02,
 author = "J.R.B. Cockett and S. Lack",
 title = "Restriction Categories I: categories of partial maps",
 journal = "Theoretical Computer Science",
 volume = 270,
 year = 2002,
 pages = {223--259}}

@article{Cockett2023,
  title={{Classical Distributive Restriction Categories}},
  author={Cockett, J. R. B. and Lemay, J.-S. P.},
  journal={Theory and Applications of Categories},
  volume={42},
  number={6},
  pages={102--144},
  year={2024}
}

@article{Cockett2009Boolean,
  title={Boolean and classical restriction categories},
  author={Cockett, J.R.B. and Manes, E.},
  journal={Mathematical Structures in Computer Science},
  volume={19},
  number={2},
  pages={357--416},
  year={2009},
  publisher={Cambridge University Press}
}

@phdthesis{Crans2004,
  title={Lie 2-algebras},
  author={Crans, A. S.},
  year={2004},
  school={University of California, Riverside}
}

@article{Stokes2024,
  title={Override and restricted union for partial functions},
  author={Stokes, T.},
  journal={Algebra universalis},
  volume={85},
  number={4},
  pages={35},
  year={2024},
  publisher={Springer}
}

@article{Cockett2012-2,
 author = "J.R.B. Cockett, X. Guo and P.J.W. Hofstra",
 title = "Range Categories II: towards regularity",
 journal = "Theory and Applications of Categories",
 volume = 26,
 year = 2012,
 pages = {453--500}}

@phdthesis{Guo2012,
 author = "X. Guo",
 title = "Products, Joins, Meets, and Ranges in Restriction Categories",
 school = "University of Calgary",
 year = 2012}

@phdthesis{Nester2024Thesis,
  title={Partial and Relational Algebraic Theories},
  author={Nester, C.},
  year={2024},
  school={Tallinn University of Technology}
}

@book{Jones1990,
  title={Systematic software development using VDM},
  author={Jones, C. B.},
  volume={2},
  year={1990},
  publisher={Prentice Hall Englewood Cliffs}
}

@misc{Prover9Mace4,
 author = "W. McCune",
 title = "{Prover9 and Mace4}",
 note = "version 2009-11A",
 year = 2009,
 url = "https://www.cs.unm.edu/~mccune/prover9/"
 }

@phdthesis{Giles2014,
 author = "B. Giles",
 title = "An investigation of some theoretical aspects of reversible computing",
 school = "University of Calgary",
 year = 2014
}

@article{Coc07,
 author = "J.R.B. Cockett and S.Lack",
 title = "Restriction Categories III: colimits, partial limits, and extensivity",
 journal = "Mathematical Structures in Computer Science",
 volume = 17,
 year = 2007,
 pages = {775--817}}

\appendix

\section{Code Listing for Example~\ref{ex:override-not-unique}}\label{app:mace4-code}
\begin{verbatim}
formulas(sos).
% monoid axioms. c(-,-) is the operation and id is the neutral element. 
c(x,c(y,z)) = c(c(x,y),z) # label(COMP_ASSOC).
c(x,id) = x # label(COMP_ID_R).
c(id,x) = x # label(COMP_ID_L).
% restriction category axioms (restriction monoid axioms).
% r(x) is the domain of definition of x.
c(r(x),x) = x # label(R1).
c(r(x),r(y)) = c(r(y),r(x)) # label(R2).
r(c(r(x),y)) = c(r(x),r(y)) # label (R3).
c(x,r(y)) = c(r(c(x,y)),x) # label(R4).
% override axioms. o(x,y) is x override y.  
o(x,o(y,z)) = o(o(x,y),z) # label(OR1).
o(x,c(r(x),y)) = x # label(OR2).
o(c(r(y),x),x) = x # label(OR3).
c(x,o(y,z)) = o(c(x,y),c(x,z)) # label(OR4).
c(r(o(x,y)),z) = o(c(r(x),z),c(r(y),z)) # label(OR5).
% override axioms again. so that b(x,y) is another override operator.
b(x,b(y,z)) = b(b(x,y),z) # label(OR1B).
b(x,c(r(x),y)) = x # label(OR2B).
b(c(r(y),x),x) = x # label(OR3B).
c(x,b(y,z)) = b(c(x,y),c(x,z)) # label(OR4B).
c(r(b(x,y)),z) = b(c(r(x),z),c(r(y),z)) # label(OR5B).
end_of_list.

formulas(goals).
% the goal is to show that our two override operators are the same.
o(x,y) = b(x,y).
end_of_list.
\end{verbatim}

\end{document}